\documentclass[11pt]{amsart}

\RequirePackage{iftex}
  \RequirePackage[T1]{fontenc}
  \RequirePackage[utf8]{inputenc}
  \RequirePackage{lmodern}
  \RequirePackage{textcomp}

\RequirePackage{microtype}

\RequirePackage[dvipsnames]{xcolor}
\definecolor{cornflowerblue}{rgb}{0.39,0.58,0.93}

\RequirePackage{mathtools}   % loads amsmath
\RequirePackage{amssymb}
\RequirePackage{amsthm}
\RequirePackage{thmtools}
\RequirePackage{thm-restate}
    \RequirePackage{stmaryrd}    % symbols like \llbracket
    \SetSymbolFont{stmry}{bold}{U}{stmry}{m}{n}
\RequirePackage{mathrsfs}    % \mathscr
\RequirePackage{esint}       % extended integral symbols
\RequirePackage{xfrac}
\RequirePackage{aligned-overset} % if used in the source
\RequirePackage{graphicx}
\RequirePackage{pdfpages}
\RequirePackage{comment}
\RequirePackage{enumerate}    % modern replacement for 'enumerate'
\RequirePackage[toc,page]{appendix}

\RequirePackage[marginparwidth=2cm]{geometry}
\RequirePackage[
  hidelinks,
  colorlinks=true,
  linkcolor=blue,
  citecolor=blue,
  backref
]{hyperref}
\RequirePackage[capitalize]{cleveref}

\RequirePackage[color=green!40,textwidth=1in]{todonotes}
\RequirePackage[normalem]{ulem}

\RequirePackage{mathtools} % loads amsmath
\RequirePackage{MnSymbol}  % defines \cupdot and \bigcupdot
\usepackage{stmaryrd}
\usetikzlibrary{arrows.meta,decorations.markings}

\tikzset{
  orientedarc/.style={
    very thick,
    postaction={decorate},
    decoration={
      markings,
      mark=at position 0.65 with {\arrow{Stealth[length=7pt,width=7pt]}}
    }
  }
}

\counterwithin{equation}{section}
\theoremstyle{definition}
\newtheorem{definition}{Definition}[section]

\newtheorem{remark}[definition]{Remark}
\newtheorem{assumption}{Assumption}
\theoremstyle{plain}
\newtheorem{theorem}[definition]{Theorem}
\newtheorem{proposition}[definition]{Proposition}

\newtheorem{lemma}[definition]{Lemma}

\newcommand\res{\mathop{\hbox{\vrule height 7pt width .3pt depth 0pt
\vrule height .3pt width 5pt depth 0pt}}\nolimits}
\def\XXint#1#2#3{{\setbox0=\hbox{$#1{#2#3}{\int}$ }
\vcenter{\hbox{$#2#3$ }}\kern-.6\wd0}}

\newcommand{\R}{\mathbb{R}}
\newcommand{\N}{\mathbb{N}}
\newcommand{\bbS}{\mathbb{S}}

\newcommand{\diam}{\mathrm{diam}}
\newcommand{\m}{\mathfrak{m}}
\newcommand{\ball}[2]{\mathbf{B}_{#2}\left(#1\right)}

\newcommand{\cH}{\mathcal{H}}

\newcommand{\spt}{\operatorname{spt}}

\newcommand{\bG}{\boldsymbol{\operatorname{G}}}

\newcommand{\bF}{\mathbf{F}}

\newcommand{\bB}{\mathbf{B}}

\newcommand{\bM}{\mathbf{M}}

\newcommand{\sF}{\mathscr{F}}

\newcommand{\graph}{\mathrm{graph}}

\def\a#1{\llbracket #1 \rrbracket}

\newcommand{\dive}{{\rm div}\,}
\newcommand{\dini}{\rm Dini}

\DeclareMathOperator*{\essinf}{ess\,inf}

    \newcommand{\half}{{\scriptscriptstyle{(1/2)}}}

\title[Flat boundary blowups]{Existence of flat blowups at boundary points of anisotropic minimizing hypercurrents}
\author{Michael Novack}
\address{University of Texas at Austin}
\email{michael.novack@math.utexas.edu}
\author{Reinaldo Resende}
\address{Florida State University}
\email{rresende@fsu.edu}

\begin{document}

\begin{abstract}
We consider $m$-dimensional currents in $\mathbb{R}^{m+1}$ that almost minimize an anisotropic energy whose integrand is continuous in the space variable and $C^{2,\dini}$ in the normal variable. We prove that if the boundary of such a current is differentiable, then the current admits a flat blowup at every boundary point. 
\end{abstract}

\maketitle
\vspace{-1cm}

\setcounter{tocdepth}{1}
\vspace{.5cm}
{  \hypersetup{linkcolor=black}
  \tableofcontents
}
\vspace{-1cm}

\section{Introduction}

In their groundbreaking work \cite{HS}, Hardt and Simon showed that boundary regularity for codimension-one area-minimizing currents is determined by their blowups. More precisely, let $T$ be an area-minimizing $m$-current in $\R^{m+1}$ with multiplicity-one boundary $\Gamma$ of class $C^{1,\alpha}$. They proved that, for every $p\in\Gamma$, $T$ admits a flat blowup at $p$ if, and only if, $T$ is regular near $p$. This equivalence reduces the boundary regularity problem to the study of blowups. They then proved that 
$$\text{every blowup of $T$ at a boundary point is flat,}$$
and hence that $T$ attaches regularly to $\Gamma$ at every point. 

For anisotropic energies, the problem has remained largely open, the basic obstruction being the absence of a monotonicity formula; see Remark \ref{remark:hardtsimon estimate} and Section \ref{remark:comparison with area}. For general boundary points, it was not known whether a flat blowup exists, nor whether the existence of one implies regularity. Indeed, Hardt and Simon \cite{HS} observe that for two-dimensional anisotropic minimizers in $\mathbb{R}^3$, ``at an arbitrary boundary point $a$, even the existence of an oriented tangent \emph{cone} is unknown,'' and we are not aware of any progress on this question. In this paper, we answer the blowup question in every dimension; namely, we prove that every boundary point of an anisotropic minimizer with multiplicity-one boundary supported on a differentiable manifold admits a flat blowup.

The existing regularity results for anisotropic minimizers are conditional on the behavior of the minimizer at the point in question. In the interior, where the theory is well developed, Schoen, Simon, and Almgren \cite{schoen1977regularity} proved that the singular set has vanishing $(m-2)$-dimensional Hausdorff measure, a bound which can be improved to $\mathcal{H}^{(m-2-\nu)}$-null for some $\nu>0$; cf. the discussion in \cite[Corollary 2.5]{figalli2017regularity}. In the other direction, Morgan \cite{morgan1991cone} produced an example of a three-dimensional minimizing cone in $\mathbb{R}^4$ with a singularity at the origin. At the boundary, Hardt \cite{hardt1977boundary} proved regularity at points where $\spt \,T$ has a supporting hyperplane, and Lin \cite{lin1985regularity}, and later Duzaar and Steffen \cite{duzaar2002optimal}, proved an $\varepsilon$-regularity theorem at points where the lower density is at most $1/2 + \varepsilon$. Our theorem, in contrast, gives an unconditional conclusion at the level of blowups: every boundary point admits a flat blowup.

We now state our main result and refer the reader to Section \ref{sec:prelim} for the precise definitions.

\begin{theorem}[Existence of a flat blowup at the boundary]\label{thm:main existence theorem intro}
    If $m\in \mathbb{N}$, $F$ is a uniformly elliptic integrand of class $C^{2,\dini}$, $T$ is an integer rectifiable $m$-current that is $(\bF,\omega)$-minimizing in $\bB_1$, and $\partial T$ is a multiplicity-one current whose support is a differentiable manifold containing $0$, then there exist $r_j\to 0$, $Q\in \mathbb{N}$, and a half-hyperplane $\Pi^+$ such that
    \begin{equation}\notag
      T_{0,r_j} \to Q\a{\Pi^+} - (Q-1) \a{-\Pi^+}\,.  
    \end{equation}
\end{theorem}

\begin{remark}[Nonautonomous integrands] 
Theorem \ref{thm:main existence theorem intro} also applies to energies with nonautonomous integrands, that is,  $F:\R^{m+1}\times\R^{m+1}\to\R$ that is continuous in the first variable and uniformly elliptic of class $C^{2,\dini}$ in the second variable. Indeed, any $(\bF, \omega)$-minimal current $T$ is $\left(\bF_0, \omega+C{\omega}_0\right)$-minimal where $F_0(\nu):=F(0, \nu)$, $\omega_0(r) := \sup_{\nu\in\bbS^m}\{|F(x,\nu) - F(y,\nu)|: |x-y|\leq r\}$, and $C>1$ is a universal constant. Since $F$ is continuous in the space variable, $\omega+C\omega_0$ is an admissible error term, and so the conclusion of Theorem \ref{thm:main existence theorem intro} holds in this setting.
\end{remark}

\begin{remark}[Higher multiplicity boundaries]
    If $\partial T$ has multiplicity $P$ for some $P>1$, then an argument of White \cite{white1983regularity} shows that $T$ decomposes on $\bB_1$ into $P$ anisotropic area minimizers $T_1,\dots, T_P$ with multiplicity one boundaries, and so Theorem \ref{thm:main existence theorem intro} applies separately to each $T_i$. If one of the $T_i$'s has a two sided blowup, then the blowup of $T$ is a two-sided plane with multiplicities $Q$ and $Q-P$ on each half plane, whereas if all $T_i$'s have one-sided blowups, this blowup will be an open book with sheets that may or may not coincide.
\end{remark}

\begin{remark}[On full boundary regularity]\label{remark:hardtsimon estimate}
When $Q=1$ in Theorem \ref{thm:main existence theorem intro}, the subsequential blowup is a half-hyperplane with multiplicity one (a ``one-sided'' blowup) and the result \cite{duzaar2002optimal} of Duzaar and Steffen applies, and $\spt \, T$ is an anisotropic minimal surface with boundary in a neighborhood of $0$. Complete boundary regularity would therefore follow from an $\varepsilon$-regularity theorem when $Q>1$ and a blowup is ``two-sided''. In the case of the area functional, the argument is rather delicate and relies in a crucial way on the remainder term in the monotonicity formula \cite{HS}. Addressing this issue and establishing the {full} boundary regularity program in the anisotropic case remains {open}.
\end{remark}

\subsection{Outline of the proof}\label{subsec: intro explanation BU argument}

The existence of nontrivial ``classical'' blowups, with no need for Preiss blowups, follows from upper mass ratio bounds and the fact that the boundary is induced by a differentiable manifold; see Remark \ref{remark:reduction to flat case}. Since the set of blowups at zero is closed under further blowups, it therefore suffices to consider the case where, up to a rotation, $T$ is an integer rectifiable $m$-current with multiplicity one boundary supported in $\R^{m-1}\times\{0\}^2 := L$ that minimizes $\bF$ in $\R^{m+1}$. The proof consists of two main steps: a barrier argument that precludes $\spt \,T$ from wrapping around $L$ infinitely often, and a blowup argument that utilizes a parametric Hopf lemma to extract a limiting tangent plane.

\subsubsection{Boundedness of the polar angle function.} Let $M$ be the oriented half-hyperplane $\{x_m>0\}\cap \{x_{m+1}=0\}$ with boundary $L$. The key quantity is a polar angle function $\boldsymbol{\theta}$ (Definition \ref{def:polar angle}), {defined on the interior regular set ${\rm Reg}_i T\subset\spt T$,} which counts how many times $T$ has wound around $L$ {in the $(x_m,x_{m+1})$-variables} (relative to a fixed base point). The reader may keep in mind the picture of $T\subset \mathbb{R}^2$ as a spiral winding around $L$, crossing $M$ once at each rotation and taking values in $2 \pi \mathbb{Z}$ on $M \cap {\rm Reg}_i T$; see Figure \ref{fig:spiral}. The goal in this step is to show that $\boldsymbol{\theta}$ is locally bounded, cf. Theorem \ref{lemma:trapping lemma}. This will be necessary in order to identify a candidate tangent plane for $T$ at the origin. 

\begin{figure}[htbp]
  \centering
\begin{tikzpicture}[scale=.8]
  \filldraw[fill=gray!20, draw=black] (1,0) circle (1);
  \node at (1.3,0.55) {$U$};

  \draw[very thick, dash pattern=on 6pt off 3pt] (0,0) -- (1,0);
  \draw[dotted, thick, ->] (1,0) -- (3.8,0);

  % spiral from p_0 to p_2 (solid)
  \draw[thick, samples=240, domain=0:720, variable=\t]
    plot ({(3 - 3*\t/1080) * cos(\t)}, {(3 - 3*\t/1080) * sin(\t)});
  % last loop from p_2 to the origin (dash-dot)
  \draw[thick, dash dot, samples=120, domain=720:1080, variable=\t]
    plot ({(3 - 3*\t/1080) * cos(\t)}, {(3 - 3*\t/1080) * sin(\t)});
  \node at (2.45,2.3) {$T$};

  \fill (3,0) circle (1.4pt) node[below right, xshift=4pt] {$p_0$};
  \fill (2,0) circle (1.4pt) node[below right, xshift=4pt] {$p_1$};
  \fill (1,0) circle (1.4pt) node[below right, xshift=3pt] {$p_2$};
  \fill (0,0) circle (1.4pt) node[above left, inner sep=1pt] {$L$};
\end{tikzpicture}
  \caption{At the base point $p_0$, $\boldsymbol{\theta}(p_0)=0$. Then $\boldsymbol{\theta}(p_1)=2\pi$, $\boldsymbol{\theta}(p_2)=4\pi$. The cross section of the barrier set $U\subset \{x_m>0\}$ through any two-dimensional slice perpendicular to $L$ over $x\in \bB_1 \cap L$ is the gray disk tangent to $L$ shown above; see Figure \ref{fig:barrier set} below to see this slice relative to the whole picture. Since that disk has positive curvature, it acts as a barrier forcing an area minimizing surface to connect $p_2$ to $L$ through a path contained in $U$ rather than exiting $U$ along the alternating dotted/dashed spiral; in $\mathbb{R}^2$, this path is the bold dashed segment shown above. }
  \label{fig:spiral}
\end{figure}
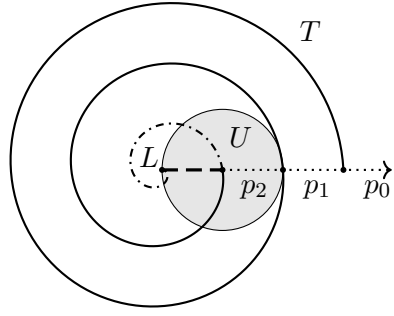

Let us consider bounding $\boldsymbol{\theta}$ from above on $\bB_1 \cap {\rm Reg}_i T$; the lower bound argument is entirely analogous. Such an upper bound would follow from showing that for some large $N\in \mathbb{N}$, the superlevel set $\{ \boldsymbol{\theta}(x)> 2\pi N \}\cap \bB_1$ is contained in the halfspace $\{x_m>0\}$, since this implies that $\boldsymbol{\theta}\leq 2\pi N + \pi/2$ on $\bB_1 \cap {\rm Reg}_i T$. In turn, to prove this halfspace containment, we use a barrier set $U$ such that
\begin{equation}\label{eq: u containment}
\{ \boldsymbol{\theta}(x)> 2\pi N \}\cap \bB_1\subset U \cap \bB_1\subset \{x_m>0\}\qquad \mbox{for some large $N\in \mathbb{N}$}\,.    
\end{equation}
We have thus reduced bounding $\boldsymbol{\theta}$ from above to the construction of a barrier set $U$ satisfying \eqref{eq: u containment}.

To choose $N$ and prove \eqref{eq: u containment}, let $N\in\N$ and $T_N$ be the current induced by ${\bB_2 \cap }\{\boldsymbol{\theta} > 2\pi N\}$. We first show that for large $N$, the support of $T_N$ is contained in an {arbitrarily} small neighborhood of $L$. This follows from lower density bounds by arguing that, if $T$ had wound around $L$ too many times at positive distance from $L$, the mass of $T$ would be infinity. We then construct $U$ so that $U \cap \bB_1$ is the cylinder $\{|(x_m,x_{m+1})-(\delta,0)|<\delta\}$ contained in $\{x_m>0\}$ and tangent to $L$, and $U\setminus \bB_2$ is a small dilation of this cylinder; see Figure \ref{fig:barrier set}. Thus $U$ is $\bF$-mean convex. In fact, $U$ is also mean convex with respect to the reflected anisotropy $G(\nu)=F(-\nu)$, which will be necessary in order to address possible mismatches in orientation at touching points. Furthermore, for large enough $N$, $\spt\, T_N$ is contained in an arbitrarily small neighborhood of $U$ and $\spt \partial T_N \subset\overline U$, and so a maximum principle (cf. Proposition \ref{corollary:mean convex barriers}) implies that $\spt\, T_N$ must be contained in $\overline U$.

\begin{figure}[htbp]
  \centering
\begin{tikzpicture}[scale=1.8]
  \def\R{0.5}    % base radius on [-1,1]
  \def\eps{0.2}  % dilation at the ends: phi(2) = 1 + eps
  \def\e{0.3}    % oblique factor: depth z is drawn as an x-shift of e*z
  \def\yb{-0.1}  % height of the bottom curve at x = +/-2
  \def\sag{0.2}  % sag of the bottom curve at x = 0
  \def\D{0.5}    % max depth of the surface's fibers (at x = +/-2)
  \def\xa{1.0}   % x from which the fibers start to curl up over L
  \def\a{0.33}   % strength of the curl at x = 2 (keep lobe under the wall)
  \pgfmathsetmacro{\Rc}{\R*(1+\eps)}             % cap radius
  \pgfmathsetmacro{\ec}{\e*\Rc}                  % seam x-radius
  \pgfmathsetmacro{\ed}{\e*\R}                   % x-radius of the x=0 slice
  \pgfmathsetmacro{\xe}{2 + sqrt(\Rc^2 - \R^2)}  % where L exits the caps

  \tikzset{declare function={
    % wall profile, bottom curve, curl amplitude
    phi(\x)   = 1 + \eps*(3*max(0,abs(\x)-1)^2 - 2*max(0,abs(\x)-1)^3);
    bot(\x)   = \yb - \sag*(1-(\x/2)^2);
    amp(\x)   = \a*(3*max(0,(\x-\xa)/(2-\xa))^2 - 2*max(0,(\x-\xa)/(2-\xa))^3);
    % fiber over x, parameter s in [0,1]: starts on L, ends on the bottom curve
    X(\x,\s)  = \x + \e*\D*(\x/2)*sin(180*\s);
    Y(\x,\s)  = \R + \s*(bot(\x)-\R) + amp(\x)*sin(180*\s);
    % parameter where the fiber is highest (0 if it never rises above L)
    ratio(\x) = min(1, (\R - bot(\x))/(amp(\x)*pi + 0.0001));
    sstar(\x) = acos(ratio(\x))/180;
  }}
  \pgfmathsetmacro{\sst}{sstar(2)}

  % hidden halves of the seams
  \draw[densely dotted, gray]
    (-2,\Rc) arc[start angle=90, end angle=270, x radius=\ec, y radius=\Rc];
  \draw[densely dotted, gray]
    (2,\Rc)  arc[start angle=90, end angle=270, x radius=\ec, y radius=\Rc];

  % ---- x = 0 slice of the cylinder: far half (behind T) ----
  \fill[gray!55, fill opacity=0.6]
    (0,\R) arc[start angle=90, end angle=270, x radius=\ed, y radius=\R] -- cycle;
  \draw[densely dotted, gray]
    (0,\R) arc[start angle=90, end angle=270, x radius=\ed, y radius=\R];

  % ---- the surface T ----
  \fill[gray!30, fill opacity=0.7]
    (-2,\R) -- (\xa,\R)
    -- plot[variable=\x, domain=\xa:2, samples=40]           % fold silhouette
         ({X(\x,sstar(\x))}, {Y(\x,sstar(\x))})
    -- plot[variable=\s, domain=\sst:1, samples=40, smooth]  % right fiber, down
         ({X(2,\s)}, {Y(2,\s)})
    -- plot[variable=\x, domain=2:-2, samples=60]            % bottom curve
         (\x, {bot(\x)})
    -- plot[variable=\s, domain=1:0, samples=40, smooth]     % left fiber, up
         ({X(-2,\s)}, {Y(-2,\s)})
    -- cycle;
  % interior fibers (thin, light); the one at 1.85 hooks over L
  \foreach \xf in {-1.5, -0.75, 0.75, 1.5, 1.85}
    \draw[very thin, gray!60, variable=\s, domain=0:1, samples=40, smooth]
      plot ({X(\xf,\s)}, {Y(\xf,\s)});
  % fold silhouette above L (not a boundary, so grey)
  \draw[thin, gray!70, variable=\x, domain=\xa:2, samples=40]
    plot ({X(\x,sstar(\x))}, {Y(\x,sstar(\x))});
  % right side edge: top part hidden behind the curl, rest visible
  \draw[thin, dash dot, gray, variable=\s, domain=0:\sst, samples=20, smooth]
    plot ({X(2,\s)}, {Y(2,\s)});
  \draw[thin, dash dot, variable=\s, domain=\sst:1, samples=40, smooth]
    plot ({X(2,\s)}, {Y(2,\s)});
  % left side edge
  \draw[thin, dash dot, variable=\s, domain=0:1, samples=40, smooth]
    plot ({X(-2,\s)}, {Y(-2,\s)});
  % bottom edge (thick) and its label
  \draw[thick, variable=\x, domain=-2:2, samples=60] plot (\x, {bot(\x)});
  \draw[very thin] ({0.5}, {bot(0.5)}) -- (1.0,-0.85)
    node[below right, inner sep=1pt] {$\{\boldsymbol{\theta} = 2\pi N\}$};

  % label for the x = 0 slice
  \draw[very thin] (-0.04,-0.4) -- (-0.9,-0.85)
    node[below left, inner sep=1pt] {2D slice of $U$ from Figure~\ref{fig:spiral}};

  % ---- x = 0 slice: near half (in front of T) ----
  \fill[gray!55, fill opacity=0.6]
    (0,\R) arc[start angle=90, end angle=-90, x radius=\ed, y radius=\R] -- cycle;
  \draw (0,\R) arc[start angle=90, end angle=-90, x radius=\ed, y radius=\R];

  % labels for the surface and the capsule
  \node at (0.45,0.1) {$T$};
  \node at (-2.35,-0.1) {$U$};

  % hidden parts of L (inside the body)
  \draw[dashed, gray] (-\xe,\R) -- (-1,\R);
  \draw[dashed, gray] ( 1,\R) -- ( \xe,\R);

  % dilated cylinder walls
  \draw[samples=120, domain=-2:2, variable=\x] plot (\x, { \R*phi(\x)});
  \draw[samples=120, domain=-2:2, variable=\x] plot (\x, {-\R*phi(\x)});

  % hemispherical caps of radius R(1+eps)
  \draw ( 2, \Rc) arc[start angle=90,  end angle=-90, radius=\Rc];
  \draw (-2,-\Rc) arc[start angle=270, end angle=90,  radius=\Rc];

  % visible halves of the seams
  \draw (-2,\Rc) arc[start angle=90, end angle=-90, x radius=\ec, y radius=\Rc];
  \draw (2,\Rc)  arc[start angle=90, end angle=-90, x radius=\ec, y radius=\Rc];

  % visible parts of L, with arrowheads at both ends
  \draw[<-] (-3.2,\R) -- (-\xe,\R);
  \draw (-1,\R) -- (1,\R);
  \draw[->] ( \xe,\R) -- (3.2,\R) node[right] {$L = \mathbb{R}^{m-1} \times \{0\}^2$};

  % tick marks on L
  \foreach \x in {-1, 1}
    \draw (\x,\R-0.08) -- (\x,\R+0.08) node[above] {$\x$};
  \foreach \x in {-2, 2}
    \draw (\x,\R-0.08) -- (\x,\R+0.08) node[above, yshift=3pt] {$\x$};
\end{tikzpicture}
  \caption{Here the plane containing the page is $\{x_m=0\}$, so that the half plane corresponding to angle $0$ and containing the bold level set $\{\boldsymbol{\theta}=2\pi N\}$ is the portion below $L$. The anisotropic area minimizer induced by $\bB_2 \cap \{\boldsymbol{\theta}>2\pi N\}$ is the gray surface $T_N$, and it is bounded by four curves: a portion of $L$, $\{\boldsymbol{\theta}=2\pi N\}$, and the alternating dotted/dashed curves $\{\boldsymbol{\theta}>2\pi N\} \cap \bB_2$ coming from slicing $T$ by $\partial \bB_2$. The barrier set $U$ is built so that it contains these four curves, and thus the boundary of $T_N$. This allows for a maximum principle argument yielding $\spt T_N \subset U$.}
  \label{fig:barrier set}
\end{figure}
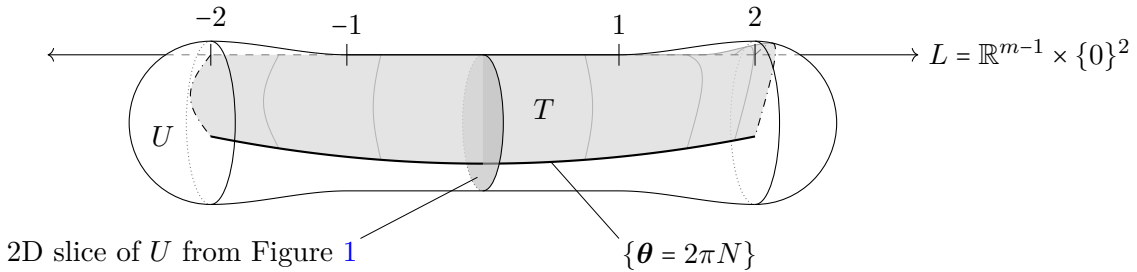

\subsubsection{Parametric Hopf lemma and blowups}
Once we know that $\boldsymbol{\theta}$ is bounded on $\bB_1$, we have a natural candidate for the limiting halfplane(s) hopefully contained in a flat blowup: the half plane $M_{\rm min}$ with boundary $L$ corresponding to the angle
\begin{equation}\notag
   \boldsymbol{\theta}_{\rm min} = \lim_{r\searrow 0}\inf \{\boldsymbol{\theta}(x) : x\in \bB_r \cap {\rm Reg}_i T  \}=:\lim_{r\searrow 0}\boldsymbol{\theta}_r \,.
   \end{equation}
We emphasize that the infimums exist for each $r\in (0,1)$ only because $\boldsymbol{\theta}$ is bounded, and that $\boldsymbol{\theta}_{\rm min}$ exists since it is the limit of a bounded, monotone sequence. To simplify the notation, let us assume that $\boldsymbol{\theta}_{\rm min}\in 2\pi \mathbb{Z}$. {We also denote by $M_r$ the halfplanes corresponding to angles $\boldsymbol{\theta}_r$.}

The first step in the identification of a flat blowup is showing that $M_{\rm min}$ must be contained in the support of any blowup $T_0$ of $T$ at $0$. Postponing the summary of this argument for a moment, the remainder of the argument is an iterative procedure where we take further blowups, repeating at each stage the argument that yielded $M_{\rm min}$ to identify further halfplanes contained in these blowups. By upper density bounds on minimizers, this process must terminate after finitely many iterations, resulting in a final blowup given by a union of halfplanes with boundary $L$. By a two-dimensional argument using the $\bF$-minimality of the final blowup and convexity of $F$ (cf. Lemma \ref{lemma: resolution in 2d}), these half hyperplanes must be all contained in the same plane, and so that plane supports a flat blowup of $T$ at $0$. 

The proof that $M_{\rm min}$ is contained in the support of any blowup $T_0$ of $T$ at $0$ consists of two steps.  First, assuming for contradiction that $M_{\rm min}\not\subset \spt T_0$, we may choose $x_0 \in M_{\rm min} \setminus \spt T_0$ and a small ball $\bB_{{r_0}}(x_0)$ {disjoint from} $\spt T_0$. We now argue that  {there exists a connected component} of ${\spt T_0 \cap \{x_m>0 \}}$ contained in the wedge $W={\{x_m>0, x_{m+1}>0 \}}$, which heuristically follows from the definition of $\boldsymbol{\theta}_{\rm min}\in 2\pi \mathbb{Z}$. Indeed, {since $\boldsymbol{\theta}_r$ is the infimum of $\boldsymbol{\theta}$ on $\bB_r$, for small $r>0$, the component of $\spt T_{0,r} \cap \bB_1 \cap \{ x_m>0\}$ achieving $\boldsymbol{\theta}_r$ cannot cross the halfplane corresponding to $\boldsymbol{\theta}_r$ (as it would thus take values strictly smaller than $\boldsymbol{\theta}_r$), and so taking a limit in these components yields a component of $\spt T_0 \cap \{x_m>0\}$ contained in $W$.  Second, using the ``separation'' between ${T_0}$ and $M_{\rm min}$ given by $\bB_{r_0}(x_0)$ and the fact that the support of $T$ is contained in the wedge $W$, {for small $r>0$ we insert the graph of a barrier function $u_r$ in between $M_{r}$ and $\spt\,T_{0,r} \cap \{x_m>0\} \cap \bB_1$ for small $r$.} By a parametric Hopf lemma (Proposition \ref{lemma:hopf currents}), {the graph of $u_r$ meets $L$ at a small positive angle $\delta$ relative to the halfplane corresponding to $\boldsymbol{\theta}_r$, and a compactness argument shows that $\delta$ depends on $r_0$ but not on $r$; see Figure \ref{fig:hopf figure}. Thus, for $r$ small enough, the parametric Hopf lemma also gives that the component of $\spt T_{0,r} \cap \bB_1 \cap \{x_m>0\}$ achieving $\boldsymbol{\theta}_r$ lives above the plane determined by $\boldsymbol{\theta}_r$. This contradicts the definition of $\boldsymbol{\theta}_r$ and completes the proof that $M_{\rm min}\subset \spt T_0$.}

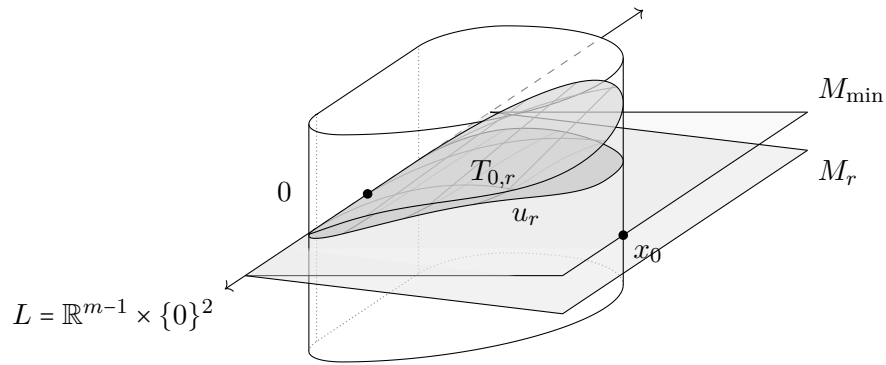
\begin{figure}[htbp]
  \centering
\begin{tikzpicture}[scale=1.5]
  \def\kx{0.45}  \def\kz{0.30}  % oblique view: (x,y,z) -> (x + kx*y, z + kz*y)
  \def\a{0.5}   % corner rounding: x-radius of the corner arcs
  \def\b{0.6}   % corner rounding: y-radius (Omega reaches |y| = 1+b)
  \def\c{1.6}   % x-radius of the right-hand side (Omega's width is a+c)
  \def\eps{0.12}  % tilt of the second half-plane z = -eps*x (exaggerated)
  \def\X{2.8}     % how far the half-planes are drawn in x
  \def\Yf{-2.4}   \def\Yb{2.4}   % y-extent of the half-planes
  % the surface T: graph of ft(x,y) over all of Omega
  \def\m{0.577}   % slope in x at the segment (tan 30 degrees)
  \def\cq{0.17}   % how quickly the slope decays with x
  \def\tw{0.05}   % twist: lower at the front (y<0), higher at the back (y>0)
  % the graph of u_r: between M_r and T
  \def\mu{0.268}  % slope in x at the segment (tan 15 degrees)
  \def\cu{0.10}   % decay of the slope with x
  \def\bu{0.03}   % ridge along y=0: falls off toward the front and back
  % the ball B_r
  \def\rb{0.35}   % radius
  \pgfmathsetmacro{\req}{0.35*\rb}   % apparent half-height of its equator

  \tikzset{declare function={
    px(\x,\y) = \x + \kx*\y;
    pz(\y,\z) = \z + \kz*\y;
    % boundary of Omega in three curved pieces (plus the flat segment x=0, |y|<=1)
    cxT(\f) = \a + \a*cos(\f);   cyT(\f) =  1 + \b*sin(\f);   % top-left corner,    f: 180 -> 90
    cxR(\f) = \a + \c*cos(\f);   cyR(\f) = (1+\b)*sin(\f);    % right side,         f: 90 -> -90
    cxB(\f) = \a + \a*cos(\f);   cyB(\f) = -1 + \b*sin(\f);   % bottom-left corner, f: 270 -> 180
    % the surfaces, and the extent of Omega along y = const (|y|<=1) or x = const (x>=a)
    ft(\x,\y) = \m*\x - \cq*\x^2 + \tw*\y*\x^2;
    fu(\x,\y) = \mu*\x - \cu*\x^2 - \bu*\x^2*\y^2;
    xw(\y)    = \a + \c*sqrt(1 - (\y/(1+\b))^2);
    yw(\x)    = (1+\b)*sqrt(1 - ((\x-\a)/\c)^2);
  }}
  % parameters of the two vertical silhouette lines
  \pgfmathsetmacro{\fF}{180 + atan(\kx*\b/\a)}      % front silhouette, on the bottom-left corner
  \pgfmathsetmacro{\fB}{atan(\kx*(1+\b)/\c)}        % back silhouette, on the right side
  % where L disappears behind the body and where it re-emerges
  \pgfmathsetmacro{\yE}{px(cxB(\fF),cyB(\fF))/\kx}
  \pgfmathsetmacro{\psiX}{acos((\kx - \kz*\a)/(\kz*\c))}
  \pgfmathsetmacro{\yX}{px(cxR(\psiX),cyR(\psiX))/\kx}
  \pgfmathsetmacro{\Ym}{-(1+\b)}       % front extreme of the body: planes split here
  \pgfmathsetmacro{\zX}{-\eps*\X}      % height of the tilted plane at x = X
  % picture coordinates of the right-hand silhouette line, and where the far edge of M_min crosses it
  \pgfmathsetmacro{\xs}{px(cxR(\fB),cyR(\fB))}
  \pgfmathsetmacro{\zs}{\kz*(\xs - \X)/\kx}

  % ---- body: silhouette region, mostly opaque ----
  \fill[white, fill opacity=0.8]
    plot[variable=\f, domain=\fB:90, samples=40]    ({px(cxR(\f),cyR(\f))},{pz(cyR(\f),1)})
    -- plot[variable=\f, domain=90:180, samples=30] ({px(cxT(\f),cyT(\f))},{pz(cyT(\f),1)})
    -- ({px(0,-1)},{pz(-1,1)})
    -- plot[variable=\f, domain=180:\fF, samples=20] ({px(cxB(\f),cyB(\f))},{pz(cyB(\f),1)})
    -- ({px(cxB(\fF),cyB(\fF))},{pz(cyB(\fF),-1)})
    -- plot[variable=\f, domain=\fF:270, samples=30]  ({px(cxB(\f),cyB(\f))},{pz(cyB(\f),-1)})
    -- plot[variable=\f, domain=-90:\fB, samples=50] ({px(cxR(\f),cyR(\f))},{pz(cyR(\f),-1)})
    -- cycle;

  % hidden part of the bottom outline
  \draw[densely dotted, gray]
    plot[variable=\f, domain=\fB:90, samples=40]    ({px(cxR(\f),cyR(\f))},{pz(cyR(\f),-1)})
    -- plot[variable=\f, domain=90:180, samples=30] ({px(cxT(\f),cyT(\f))},{pz(cyT(\f),-1)})
    -- ({px(0,-1)},{pz(-1,-1)})
    -- plot[variable=\f, domain=180:\fF, samples=20] ({px(cxB(\f),cyB(\f))},{pz(cyB(\f),-1)});

  % hidden vertical edges of the flat face
  \draw[densely dotted, gray] ({px(0,-1)},{pz(-1,-1)}) -- ({px(0,-1)},{pz(-1,1)});
  \draw[densely dotted, gray] ({px(0, 1)},{pz( 1,-1)}) -- ({px(0, 1)},{pz( 1,1)});

  % hidden stretch of L
  \draw[dashed, gray] ({px(0,\yE)},{pz(\yE,0)}) -- ({px(0,\yX)},{pz(\yX,0)});

  % ---- half-planes: everything from the body's front backwards (through the translucent body) ----
  \fill[gray!12, fill opacity=0.4]
    ({px(0,\Ym)},{pz(\Ym,0)}) -- ({px(\X,\Ym)},{pz(\Ym,0)})
    -- ({px(\X,\Yb)},{pz(\Yb,0)}) -- ({px(0,\Yb)},{pz(\Yb,0)}) -- cycle;
  \draw[thin] ({px(\X,\Yb)},{pz(\Yb,0)}) -- ({px(0,\Yb)},{pz(\Yb,0)});
  \fill[gray!25, fill opacity=0.4]
    ({px(0,\Ym)},{pz(\Ym,0)}) -- ({px(\X,\Ym)},{pz(\Ym,\zX)})
    -- ({px(\X,\Yb)},{pz(\Yb,\zX)}) -- ({px(0,\Yb)},{pz(\Yb,0)}) -- cycle;
  \draw[thin] ({px(\X,\Yb)},{pz(\Yb,\zX)}) -- ({px(0,\Yb)},{pz(\Yb,0)});

  % visible part of the bottom outline
  \draw
    plot[variable=\f, domain=\fF:270, samples=30]   ({px(cxB(\f),cyB(\f))},{pz(cyB(\f),-1)})
    -- plot[variable=\f, domain=-90:\fB, samples=50] ({px(cxR(\f),cyR(\f))},{pz(cyR(\f),-1)});

  % vertical silhouette lines
  \draw ({px(cxB(\fF),cyB(\fF))},{pz(cyB(\fF),-1)}) -- ({px(cxB(\fF),cyB(\fF))},{pz(cyB(\fF),1)});
  \draw ({px(cxR(\fB),cyR(\fB))},{pz(cyR(\fB),-1)}) -- ({px(cxR(\fB),cyR(\fB))},{pz(cyR(\fB),1)});

  % top face outline (all visible)
  \draw
    ({px(0,-1)},{pz(-1,1)}) -- ({px(0,1)},{pz(1,1)})
    -- plot[variable=\f, domain=180:90, samples=30]  ({px(cxT(\f),cyT(\f))},{pz(cyT(\f),1)})
    -- plot[variable=\f, domain=90:-90, samples=60]  ({px(cxR(\f),cyR(\f))},{pz(cyR(\f),1)})
    -- plot[variable=\f, domain=270:180, samples=30] ({px(cxB(\f),cyB(\f))},{pz(cyB(\f),1)})
    -- cycle;

  % ---- the graph of u_r (below T, so drawn first) ----
  \fill[gray!55, fill opacity=0.5]
    ({px(0,-1)},{pz(-1,0)}) -- ({px(0,1)},{pz(1,0)})
    -- plot[variable=\f, domain=180:90, samples=30]
         ({px(cxT(\f),cyT(\f))},{pz(cyT(\f),fu(cxT(\f),cyT(\f)))})
    -- plot[variable=\f, domain=90:-90, samples=60]
         ({px(cxR(\f),cyR(\f))},{pz(cyR(\f),fu(cxR(\f),cyR(\f)))})
    -- plot[variable=\f, domain=270:180, samples=30]
         ({px(cxB(\f),cyB(\f))},{pz(cyB(\f),fu(cxB(\f),cyB(\f)))})
    -- cycle;
  \foreach \xc in {0.5, 1.0, 1.5}
    \draw[very thin, gray!75, variable=\y, domain=-yw(\xc):yw(\xc), samples=30]
      plot ({px(\xc,\y)},{pz(\y,fu(\xc,\y))});
  \draw[thin]
    plot[variable=\f, domain=180:90, samples=30]
      ({px(cxT(\f),cyT(\f))},{pz(cyT(\f),fu(cxT(\f),cyT(\f)))})
    -- plot[variable=\f, domain=90:-90, samples=60]
      ({px(cxR(\f),cyR(\f))},{pz(cyR(\f),fu(cxR(\f),cyR(\f)))})
    -- plot[variable=\f, domain=270:180, samples=30]
      ({px(cxB(\f),cyB(\f))},{pz(cyB(\f),fu(cxB(\f),cyB(\f)))});

  % ---- the surface T: graph over Omega, boundary = segment on L + curve on the lateral wall ----
  \fill[gray!30, fill opacity=0.7]
    ({px(0,-1)},{pz(-1,0)}) -- ({px(0,1)},{pz(1,0)})
    -- plot[variable=\f, domain=180:90, samples=30]
         ({px(cxT(\f),cyT(\f))},{pz(cyT(\f),ft(cxT(\f),cyT(\f)))})
    -- plot[variable=\f, domain=90:-90, samples=60]
         ({px(cxR(\f),cyR(\f))},{pz(cyR(\f),ft(cxR(\f),cyR(\f)))})
    -- plot[variable=\f, domain=270:180, samples=30]
         ({px(cxB(\f),cyB(\f))},{pz(cyB(\f),ft(cxB(\f),cyB(\f)))})
    -- cycle;
  \foreach \yc in {-1, 0, 1}
    \draw[very thin, gray!60, variable=\x, domain=0:xw(\yc), samples=30]
      plot ({px(\x,\yc)},{pz(\yc,ft(\x,\yc))});
  \foreach \xc in {0.7, 1.2, 1.7}
    \draw[very thin, gray!60, variable=\y, domain=-yw(\xc):yw(\xc), samples=30]
      plot ({px(\xc,\y)},{pz(\y,ft(\xc,\y))});
  \draw[thin]
    plot[variable=\f, domain=180:90, samples=30]
      ({px(cxT(\f),cyT(\f))},{pz(cyT(\f),ft(cxT(\f),cyT(\f)))})
    -- plot[variable=\f, domain=90:-90, samples=60]
      ({px(cxR(\f),cyR(\f))},{pz(cyR(\f),ft(cxR(\f),cyR(\f)))})
    -- plot[variable=\f, domain=270:180, samples=30]
      ({px(cxB(\f),cyB(\f))},{pz(cyB(\f),ft(cxB(\f),cyB(\f)))});
  \draw ({px(0,-1)},{pz(-1,0)}) -- ({px(0,1)},{pz(1,0)});   % common boundary of T and u_r on L

  % labels for the surfaces (u_r's label hangs just below its front edge, to the right)
  \node at ({px(1.45,-0.75)},{pz(-0.75,ft(1.45,-0.75))}) {$T_{0,r}$};
  \node[below] at ({px(cxR(-35),cyR(-35))},{pz(cyR(-35),fu(cxR(-35),cyR(-35)))}) {$u_r$};

  % ---- the ball B_r centred at the origin, with an equator for the 3D look ----
  %\draw[thin] (0,0) circle (\rb);
  %\draw[thin] (-\rb,0) arc[start angle=180, end angle=360, x radius=\rb, y radius=\req];
  %\draw[thin, densely dotted, gray] (-\rb,0) arc[start angle=180, end angle=0, x radius=\rb, y radius=\req];
  \fill (0,0) circle (1.2pt);   % the origin, centre of B_r
  \node[left] at (-0.58,0.02) {$0$};

  % ---- half-planes: front strips (in front of the body) and their visible edges ----
  \fill[gray!12, fill opacity=0.4]
    ({px(0,\Yf)},{pz(\Yf,0)}) -- ({px(\X,\Yf)},{pz(\Yf,0)})
    -- ({px(\X,\Ym)},{pz(\Ym,0)}) -- ({px(0,\Ym)},{pz(\Ym,0)}) -- cycle;
  \draw[thin] ({px(0,\Yf)},{pz(\Yf,0)}) -- ({px(\X,\Yf)},{pz(\Yf,0)}) -- ({px(\X,\Yb)},{pz(\Yb,0)});
  \fill[gray!25, fill opacity=0.4]
    ({px(0,\Yf)},{pz(\Yf,0)}) -- ({px(\X,\Yf)},{pz(\Yf,\zX)})
    -- ({px(\X,\Ym)},{pz(\Ym,\zX)}) -- ({px(0,\Ym)},{pz(\Ym,0)}) -- cycle;
  \draw[thin] ({px(0,\Yf)},{pz(\Yf,0)}) -- ({px(\X,\Yf)},{pz(\Yf,\zX)}) -- ({px(\X,\Yb)},{pz(\Yb,\zX)});

  % the point x_0: where the outline of the cylinder crosses the edge of M_min
  \fill (\xs,\zs) circle (1.2pt) node[below right] {$x_0$};

  % visible parts of L
  \draw[<-] ({px(0,-2.8)},{pz(-2.8,0)}) -- ({px(0,\yE)},{pz(\yE,0)});
  \draw[->] ({px(0,\yX)},{pz(\yX,0)}) -- ({px(0,5.4)},{pz(5.4,0)});
  \node at ({px(0,-2.8)},{pz(-2.8,0)}) [below left] {$L = \mathbb{R}^{m-1} \times \{0\}^2$};

  % labels for the half-planes, at their back-right (top-right in the picture) corners
  \node[above right] at ({px(\X,\Yb)},{pz(\Yb,0)})   {$M_{\mathrm{min}}$};
  \node[below right] at ({px(\X,\Yb)},{pz(\Yb,\zX)}) {$M_r$};
\end{tikzpicture}
  \caption{The barrier function $u_r$ is constructed on a cylinder contained in $\{x_m>0\}$, meets $L$ with a positive angle above $M_r$, and is squeezed in between $M_r$ and $T_{0,r}$. The positive angle $\delta$ depends on the distance of $x_0$ to the blowup.}
  \label{fig:hopf figure}
\end{figure}

\subsection{Further discussion}\label{remark:comparison with area}
The idea of using a maximum principle argument nearby a point of minimum (or maximum) angle $\boldsymbol{\theta}$ to deduce flatness is due to Hardt and Simon \cite{HS}, but they use a completely different argument to control $\boldsymbol{\theta}$. In the isotropic case, the monotonicity formula and resulting \emph{conical} blowups imply that $\boldsymbol{\theta}$ is positively one homogeneous, and in particular is characterized by its values on the link of a blowup cone. By an induction on dimension {using their $\varepsilon$-regularity theorem and the conical structure}, the link is regular up to the boundary, and so $\boldsymbol{\theta}$ is continuous on the link and thus achieves its maximum. For the anisotropic problem considered here, blowups are not cones a priori {and there is not a two-sided $\varepsilon$-regularity theorem}, so $\boldsymbol{\theta}$ is not readily bounded by using continuity and compactness on the link. Therefore, we instead introduce the new barrier set to bound the polar angle function and circumvent the lack of a monotonicity formula {and $\varepsilon$-regularity theorem}.

There is also the anisotropic halfspace Bernstein theorem of Du, Mooney, Yang, and Zhu \cite{du2023half}. The authors prove that any smooth solution to the anisotropic minimal surface equation with linear boundary data on a halfspace (or even a convex domain that is not the entire space) is linear. In this graphical setting, $\boldsymbol{\theta}$ automatically takes values in $[-\pi/2,\pi/2]$ (they work instead with slopes of linear functions), and a tilting planes argument combined with ideas from nonlinear PDE theory allows them to prove that the entire configuration is flat. 

\subsection{Acknowledgments} MN has been supported by NSF Grant DMS-2607241. Figures \ref{fig:spiral}--\ref{fig:hopf figure} were drawn in TikZ with the assistance of publicly available LLMs.

\section{Notation and background}\label{sec:prelim}

We use $\ball{p}{r}$ for the open ball of radius $r$ and centered at $p$. {If $p=0$, we simply write $\bB_r$.} We use the notation \(p = (p^{\prime},p_{m+1})\in\R^m\times \R\) and $B_r(p^\prime)$ for the open ball in $\R^m$ of radius $r$ and centered at $p^\prime$.

For any subset $V$ of $\R^{m+n}$, define its $\epsilon$-neighborhood by $B_\epsilon(V):=\{x\in\R^{m+n}: {\rm dist}\,(x,V) <\epsilon\}$. We denote the $k$-dimensional Hausdorff measure in $\R^{m+n}$ by $\cH^k$.  

For background on the theory of currents, we refer the reader to \cite{LSimon_GMT,Fed}. 

Fixing $\{e_1,\ldots,e_{m+1}\}$ to be the canonical basis of $\R^{m+1}$, the Hodge star operator will always be denoted by $\star$, and it is defined as the linear isometry characterized by $\star(e_1\wedge\ldots \widehat{e_i}\ldots\wedge e_{m+1}) = (-1)^{m+1-i}e_i$ for any $i\in\{1,\ldots, m+1\}$.

We refer the reader to \cite{maggi2012sets} for background on sets of finite perimeter. Given a set of locally finite perimeter $E\subset \R^{m+1}$, we denote its reduced boundary by $\partial^* E$, the local perimeter by $P(E;A)$ for any measurable set $A\subset \R^{m+1}$, and its outward unit normal at $x$ by $\nu_E(x)$ whenever it exists. We denote by $P_F(E;A)$ the anisotropic perimeter of $E$ relative to $A$.

Our convention for the mean curvature is such that $\partial \bB_r$ has mean curvature $1/r$.

\subsection{Anisotropic integrands and $\boldsymbol{(\bF,\omega)}$-minimizers}\label{subsec:anis integrands}

Let us set the nomenclature and hypothesis on the integrands that will be used in this work. A function \(F:\R^{m+1}\to \R_+\) is called a \emph{geometric or parametric integrand}, and the \emph{anisotropic mass of $T$}, where $T$ is an integer rectifiable $m$-current in $\R^{m+1}$, is the number
\begin{equation*}
    \bF(T) = \int_{\R^{m+1}} F(\star\vec{T}(x))d|T|(x).
\end{equation*}

Fixing the direction $e_{m+1}$, the function $F_\S:\mathbb{R}^m \to \mathbb{R}$ defined by $F_\S(v)=F((-v,1))$ is called a {\it non-parametric integrand}. Thus for an open subset $\Omega \subset \mathbb{R}^m$ and the current $\bG_u$ induced by the graph of a $W^{1,1}$ function $u:\Omega \to \mathbb{R}$ oriented by its upward pointing normal,
\begin{equation}\label{eq:calF def}
\bF(\bG_u) = \sF(u) := \int_\Omega F_\S(\nabla u)\,d\mathcal{H}^m\,.
\end{equation}

The following is the regularity class of our integrands.

\begin{definition}\label{def: unif elliptic C2 Dini}
We say that \emph{$F$ is uniformly elliptic of class $C^{2,\dini}$} if $F$ is a geometric integrand that is positively one-homogeneous, $C^2$ on $\R^{m+1}\setminus\{0\}$, and satisfies
\begin{equation*}
    a|\tau|^2
    \leq D^2F(\nu)[\tau,\tau]
    \leq A|\tau|^2,
    \qquad
    \forall\,\nu\in\mathbb S^m,\quad \tau\in\nu^\perp.
\end{equation*}
Moreover, setting $\rho_F(r)
    :=
    \sup\bigl\{
        |D^2 F(u)-D^2 F(v)|:
        u,v\in\bbS^m,\ |u-v|\leq r
    \bigr\}$,
we require $
    \int_0^1 \frac{\rho_F(r)}{r}\,dr < + \infty$. We denote the upper and lower bounds of $F$ on $\bbS^m$ by $\Lambda$ and $\lambda$, respectively.
\end{definition}

We define the interior regular set of a current. 

\begin{definition}[Regular sets]\label{def: regular set}
Let $T$ be an integer rectifiable $m$-current in $\R^{m+1}$. The interior regular set ${\rm Reg}_i T$ is the set of all points $x\in\spt T\setminus \spt\,\partial T$ for which there are $r>0$, $\beta >0$, and an embedded hypersurface $\Sigma$ of class $C^{1,\beta}$ satisfying $\Sigma\cap \bB_r(x) = \spt T\cap\bB_r(x)$.
\end{definition}

We now define almost minimality.

\begin{definition}
[$(\bF,\omega)$-minimizers]\label{def: F omega minimizer}
Let $U\subset\R^{m+1}$ be open, let $B$ be an $(m-1)$-dimensional
integral current, and let $\omega:[0,\infty)\to[0,\infty)$ be
nondecreasing with $\lim_{r\to 0}\omega(r)=0$. We say that an $m$-dimensional integral
current $T$ is an \emph{$(\bF,\omega)$-minimizer in $U$ with boundary
$B$} if $\partial T=B$ and
\begin{equation*}
    \bF(T)
    \leq
    \bF(T+X)+\omega(r)\mathbf{M}(X)
\end{equation*}
for every $r>0$ and every $m$-dimensional integer rectifiable current $X$
such that $\partial X=0,
    \spt X\subset U,$ and $
    \diam(\spt X)\leq r$. If $\omega \equiv 0$, we say that $T$ is an $\bF$-minimizer.
\end{definition}

\section{Preliminary results}

\subsection{Small data existence for the Dirichlet problem and a comparison principle}

Here we adapt the argument of \cite[Lemma 3.5]{figalli2017regularity} to show the existence and $C^{1,\alpha}$-regularity up to the boundary for the Dirichlet problem corresponding to the operator $\nabla u \mapsto {\rm div}\, \nabla F_\S( \nabla u)$. We recall for example that in the case of the minimal surface operator, mean convexity of $\Omega$ is equivalent to existence of solutions to the Dirichlet problem for arbitrary continuous boundary data, cf. \cite[Theorem 1]{jenkins1968dirichlet}. Thus some smallness assumption on the boundary data is necessary to have existence for general smooth domains (the sharp condition is essentially smallness in the Lipschitz seminorm \cite{williams1984dirichlet}). 

\begin{proposition}[Small data existence for the Dirichlet problem]\label{lemma: small data existence}
    If $F:\mathbb{R}^{m+1} \to \mathbb{R}_+$ is uniformly elliptic of class $C^{2,\dini}$, $U$ is a {smooth} domain, {$\eta \in (0,1)$}, then there are $\kappa_0>0$, $C>0$, and $\alpha\in (0,1)$ with the following property:
    
\setlength{\leftskip}{.25cm}
    \noindent if {$g\in C^\infty(\partial U)$ and} ${\|g\|_{C^{1,\eta}(\partial U)}\leq \kappa_0}$, then there is a unique ${u\in (C^{1,\alpha} \cap C^{2})(U)}$ minimizing $\sF$ on $U$ among $v\in W^{1,1}(U)$ with $v=g$ on $\partial U$ and satisfying ${\rm div}\, \nabla F_\S(\nabla u)=0$ in the weak sense with
    \begin{equation}\notag
     \|u\|_{C^{1,\alpha}(U)}\leq C {\|g\|_{C^{1,\eta}(\partial U)}}\,.   
    \end{equation}
\setlength{\leftskip}{0pt}
\end{proposition}

\begin{proof}
Let $\tilde{F}$ be a modification of $F_\S$ such that $\tilde{F}=F_\S$ on $B_{100}$ and $\tilde{F}$ satisfies
\begin{equation}\notag
   1/c < D^2 \tilde{F}(x) < c \qquad \mbox{for some $c>1$}\,.
\end{equation}
Then $\tilde{F}$ is uniformly convex with quadratic growth, so the direct method in $W^{1,2}(U)$ yields a unique minimizer of $\int_{U} \tilde{F}(\nabla u)$. {The minimality of $u$ for this modified functional is enough to apply verbatim the proof of \cite[Lemma 3.5]{figalli2017regularity} (which is stated for balls but applies to any $C^2$ domain since the boundary regularity result \cite{silvestresirakov} does), yielding the existence of $\alpha \in (0,1)$ and $C>0$ such that}
\begin{equation}\notag
    \|u \|_{C^{1,\alpha}(U)}\leq C {\| g\|_{C^{1,\eta}(\partial U)}}\,.
\end{equation}
Thus by choosing $\kappa_0$ small enough, we may ensure that $|\nabla u|< 100$ and $F_\S(\nabla u) = \tilde{F}(\nabla u)$ on $U$. As a consequence, $u$ in fact solves ${\rm div}\, \nabla F_\S(\nabla u)=0$ in the weak sense. The strict convexity of $F_\S$ now implies that $u$ minimizes $\sF$, and the proof is complete {save for the claim that $u\in C^2$. To prove the $C^2$ regularity, we first observe that by the gradient bound, the Euler Lagrange operator is uniformly elliptic, and so a standard difference quotient argument shows that $u\in W^{2,2}(U)$. Therefore, the partial derivatives of $u$ are weak solutions to the equation $\partial_i [(F_\S)_{p_ip_j}(\nabla u)\partial_j (\partial_k u) ]=0$, which is a uniformly elliptic divergence form equation with Dini coefficients. This allows us to apply the $C^1$ regularity result of \cite{Li2017DiniC1} for such equations to deduce that $\partial_k u\in C^1$ for each $k$, and so $u\in C^2$.}
\end{proof}

\subsection{Mass ratio bounds}\label{subsec: mass ratior lower bound}

The fact that $\bF$-minimizers satisfy mass ratio upper bounds is understood to be standard, but we failed to find a reference for it at boundary points. Therefore, we include a proof below.  

\begin{lemma}[Local upper $\bF$-bounds]\label{lemma:local upper bf bounds}
Under the assumptions of Theorem \ref{thm:main existence theorem intro}, for every $x\in\spt(T)$, there are $C_x$, $Q_x$, and $R(x)>0$, depending on $T$, $F$, and $x$, such that $\bF(T\res{\ball{x}{r}})\leq C_xr^m$ for all $r\in (0, R(x))$ and the multiplicity of $T$ is bounded by $Q_x$ on $\ball{x}{R(x)}$.
\end{lemma}
\begin{proof}
    We denote $\partial T \res\bB_1=\a{\Gamma}$. {If $x\in \spt T \setminus \Gamma$, then the result is standard, so we assume $x\in \Gamma$.} First we show the multiplicity bound. Decompose ${\rm Reg}_i\, T {\cap \bB_1}$ into its connected components $C_i$. Note that up to $\mathcal{H}^m$-null sets, the multiplicity function $\Theta$ is some constant $Q_i$ on each $C_i$. Since, by interior regularity, $\bF(T\res\bB_{{1}}) = \sum \bF(Q_i\a{C_i}\res \bB_{{1}})$, a standard argument using the subadditivity of $\bF$ shows that each $C_i$ is an $\bF$-minimizer. {Now for those $C_i$ such that $\bB_1 \cap \partial C_i=\varnothing$, interior lower density bounds (see e.g.~ \cite[Lemma 4]{bombieri1982regularity}) imply that $\bF(Q_i \a{C_i} \res{\bB_r(y)} \gtrsim r^m$ for any $y\in \bB_1$ and $r<1-|y|$, which implies that only finitely many such $C_i$ have support intersecting $\bB_{(1-|x|)/2}$ non-trivially. Thus, choosing $R\in (0,(1-|x|)/2)$ such that $\Gamma \cap \bB_R(x)$ is connected, the multiplicities coming from these components are bounded on $\bB_{(1-|x|)/2}$. So it remains to bound the multiplicites on the other components.}

    We next claim that for each $C_i$ with $\bB_{R}(x)\cap \partial C_i \neq\varnothing$, $\partial \a{C_i}{\res \bB_R(x)}= {\pm}\a{\Gamma}\res \bB_R(x)$. To prove this, first we observe that $\partial\a{C_i}\res{({\bB_R(x)}\setminus \Gamma)}=0$. Indeed, $\spt (\partial \a{C_i}) {\cap \bB_R(x)} \subset \spt T {\cap \bB_R(x)} = [{\rm Reg}_i\, T\cup \Gamma \cup {\rm Sing}_i\, T]{\cap \bB_R(x)}$, but for every $x\in {\rm Reg}_i\, T$, there is $\ball{x}{r_x}$ and $j(x)$ such that $T \res{\ball{x}{r_x}}=Q_{j(x)} \a{C_{j(x)}}\res{\ball{x}{r_x}}$ (in which case $\a{C_{j'}}$ is $0$ on that ball for $j'\neq j(x)$), and $\spt T \setminus ({\rm Reg}_i\,T \cup \Gamma)$ has Hausdorff dimension less than $m-2$ (by interior regularity \cite{figalli2017regularity}). So $\partial \a{C_i}{\res \bB_R(x)}$ is supported on $\Gamma$, and then by the constancy lemma (\cite[4.1.31]{Fed}) we conclude that $\partial \a{C_i} \res \bB_R(x) =k_i \a{\Gamma} \res \bB_R(x)$ for some $k_i\in \mathbb{Z}$. Furthermore, by White's decomposition \cite{white1983regularity}, since each $\a{C_i}$ is an $\bF$-minimizer and $C_i$ is connected it must be the case that each $|k_i|=1$. Indeed, if $|k_i|>1$, we could decompose $\a{C_i}$ into two nontrivial $\bF$-minimizers, $\a{C_i^1}+\a{C_i^2}$, but since $C_i$ is connected, $C_i^1$ and $C_i^2$ must coincide, which is impossible since the multiplicity of $\a{C_i}$ is one. 

    To conclude the proof that the multiplicity is bounded {on $\bB_R(x)$}, since $\Theta$ is constant on each $C_i$ {which intersects $\bB_R(x)$}, it suffices to show that there are a finite number of $C_i$'s with $\bB_R(x)\cap C_i \neq\varnothing$. Indeed, if there were infinitely many, then, using $\sum \bF(\a{C_i}\res \ball{x}{R})<\infty$ and the properties of $F$, we could find a subsequential limit $T'$ of $\{\a{C_i}\res\ball{x}{R}\}_{i\in\N}$ such that $\partial T' {\res \bB_R(x)=\pm \a{\Gamma} \res \bB_R(x)}$, and, by lower semicontinuity of the mass, $\bM(T')=0$, which is impossible. 

    With the local multiplicity bound in hand, the local upper $\bF$-bound follows from a standard competitor argument; we refer the reader to \cite[(29)-(33)]{schoen1977regularity}.
\end{proof}

\subsection{Conical $\bF$-minimizers in $\R^2$}We now show a simple lemma that follows directly from strict convexity of $F$ and shows that conical minimizers in $\R^2$ must be flat. This will be used later on.

\begin{lemma}\label{lemma: resolution in 2d}
Let $C$ be an integer rectifiable $1$-current in $\R^2$ which is a cone,
minimizes $\bF$, and $\partial C = d\a{0}$, $d\in\N\setminus\{0\}$. Then there are lines
$\ell, \ell_1,\ldots, \ell_d$ through the origin and an integer $Q\geq d$ such that 
\begin{equation}\label{eq:final C formula}
    C=Q\a{ \ell^+}+(Q-d)\a{ \ell^-}\mbox{ if }Q>d\qquad  \mbox{ or }\qquad  C = \sum_{i=1}^d \a{\ell_i^+}\mbox{ if }Q=d,
\end{equation}
where $\ell^\pm$ and $\ell_i^\pm$ are the two connected components of $\ell\setminus\{0\}$ and $\ell_i\setminus\{0\}$ oriented by a fixed orientation of $\ell$ and $\ell_i$, respectively.
\end{lemma}

\begin{proof}
Since $C$ is a $1$-dimensional integer rectifiable cone with locally finite
mass, its support is a finite union of rays from the origin. Therefore, we can write 
\begin{equation}\label{eq:C formula}
C=\sum_{i=1}^N k_i \a{ r_i}\,,    
\end{equation}
where $k_i\in \mathbb Z\setminus\{0\}$ and $r_i=\{t p_i:t\geq 0\}$ for distinct points $p_i\in \mathbb S^1$. Since $\partial C={d}\a{ 0}$ {and we may take $\a{r_i}$ to be oriented so that $\partial \a{r_i}=\a{0}$}, we must have $\sum_{i=1}^N k_i={d}$.

We claim that if $k_i>0$ and $k_j<0$, then
$
    p_j=-p_i.
$
Suppose not. Fix $R>0$. We have $
    \partial(
        \a{ [0,Rp_i]}
        +
        \a{ [Rp_j,0]}
    )
    =
    \a{ Rp_i}-\a{Rp_j}$. Then we replace it by
$
    \a{ [Rp_j,Rp_i]}\,,
$
producing an admissible competitor for $C$ in $B_R$. By strict convexity of $F$, we derive that
\[
    \mathbf F(\a{ [Rp_j,Rp_i]})
    <
    \mathbf F(\a{ [0,Rp_i]})
    +
    \mathbf F(\a{ [Rp_j,0]}).
\]
This gives a contradiction and proves
the claim.

Now, if all $k_i>0$, then $\sum_{i=1}^N k_i={d}$ {and the latter alternative in \eqref{eq:final C formula} follows from \eqref{eq:C formula}}. 
Otherwise, there is at least one negative coefficient. Since
$\sum_i k_i=d$, there is also at least one positive coefficient. By the claim,
every positively oriented ray must be opposite to every negatively oriented ray.
Because the directions $p_i$ are distinct, there can be only one positive
direction and only one negative direction. This concludes the proof of the lemma.
\end{proof}

\subsection{Comparison principle}

The following lemma is an adaptation of \cite[Lemma 2.12]{maggidephilippisyoung} (proved under $C^{2,1}$ regularity of $F$) for $C^{2,\dini}$ integrands.

\begin{lemma}[Comparison principle]\label{lemma:second comp princ}
    If $U\subset \mathbb{R}^m$ is a bounded $C^1$ domain, $F$ is a uniformly elliptic integrand of class $C^{2,\dini}$, $\alpha\in (0,1)$, $a<b\in \mathbb{R}$, $u\in C^{1,\alpha}(U)$ is a weak solution to $\dive [\nabla F_\S(\nabla u)]=0$, $E \subset U\times (a,b)$ minimizes the anisotropic perimeter in $U \times (a,b)$ with $P_F(E)<\infty$, and 
\begin{equation*}
    E^{\half} \cap \partial [U \times (a,b)]\supset \{(x',x_{m+1})\in (\partial U)\times (a,b): x_{m+1}< u(x')\} \cup (U \times \{a\}),
\end{equation*}
then
\begin{equation*}
    \big|\{(x',x_{m+1}):x'\in U,\, x_{m+1}\leq u(x')\} \setminus E\big|=0\,.
\end{equation*}
\end{lemma}

\begin{proof}
The proof of \cite[Lemma 2.12]{maggidephilippisyoung}, which is stated for $C^{2,1}$ integrands, readily works here as well. {The only place this enters in the proof is when applying the divergence theorem  to the vector field $X(x):= \nabla F(\nabla u(x'),-1)$. Since, as shown in Proposition \ref{lemma: small data existence}, $u\in C^2$, this step and the proof in \cite[Lemma 2.12]{maggidephilippisyoung} proceeds     without issue.}
\end{proof}

\section{Anisotropic mean convex barriers and a parametric Hopf lemma}

In this section, we prove variations of the maximum principle at the level of functions and of $\bF$-minimizers. These are used to construct anisotropic mean convex barriers to $\bF$-minimizers, see Proposition \ref{corollary:mean convex barriers}. We also prove a ``strong'' Hopf's lemma (Proposition \ref{lemma:hopf currents}) for currents that will play a crucial role in the proof of our main theorem. 

\subsection{Maximum principles and anisotropic mean convex barriers}

\begin{lemma}[Strong maximum principle for graphs]\label{lemma:strong max principle for graphs}
Suppose that $G$ is uniformly elliptic of class ${C^{2,\dini}}$, $\Omega\subset \mathbb{R}^m$ is a bounded, connected open set, $\beta \in (0,1]$, $u,v\in C^{1,\beta}(\Omega)$ are weak solutions to
\begin{equation}\label{eq:uv equations}
-{\rm div}\,\left[\nabla G_\S(\nabla v)\right]=0,\qquad -{\rm div}\,\left[\nabla G_\S(\nabla u)\right]=g\in L^\infty(\Omega),
\end{equation}
where $\essinf_\Omega g \geq 0$, and $u \geq v$ on $\Omega$. If in addition{
\begin{enumerate}[\upshape (i)]
\item $\essinf_\Omega g \geq \kappa>0$ for some $\kappa>0$, then $u>v$ on $\Omega$; whereas
\item if $g\equiv 0$, then either $u\equiv v$ on $\Omega$ or $u>v$ on $\Omega$.
\end{enumerate}
}\end{lemma}

\begin{proof}
For every $x\in\Omega$, we have
$\nabla G_\S(\nabla u)-\nabla G_\S(\nabla v)
=
{\rm M}(x)(\nabla u-\nabla v)$ where
\begin{equation*}
{\rm M}(x)
:=
\int_0^1
D^2G_\S\big(\nabla v(x)+t(\nabla u(x)-\nabla v(x))\big)dt.
\end{equation*}
Subtracting the two equations in \eqref{eq:uv equations}, we find
\begin{equation*}
-{\rm div}\,[{\rm M}\nabla (u-v)]=g \geq 0\quad \text{ weakly in $\Omega$}.
\end{equation*}
The matrix ${\rm M}$ is clearly locally bounded and its {uniform} ellipticity follows from uniform ellipticity of $G$ {under $L^\infty$ gradient bounds on $u$ and $v$. Note that the coefficient matrix $M$ is Dini continuous since $G$ is $C^{2,\dini}$.}

{Suppose now for contradiction that $\{x\in \Omega:u(x)=v(x)\}\neq \varnothing$ but that $\{u=v\}\neq \Omega$. Then we can find $B_r(x_0)\subset \!\subset \Omega$ such that $u>v$ on $B_r(x_0)$ and $u(y_0)=v(y_0)$ for some $y_0\in \partial B_r(x_0)$. By the Hopf lemma for divergence form equations with Dini coefficients \cite[Theorem 2.1]{apushkinskaya2019boundary}, we find that $\nabla (u-v)(y_0)\neq 0$. But together with $u(y_0)=v(y_0)$, this contradicts $u\geq v$ on $\Omega$. So we either have $u>v$ on $\Omega$ or $u\equiv v$ on $\Omega$, and (i)-(ii) follow immediately.}\end{proof}

The next item is a strong maximum principle for $\bF$-minimizers, similar to the statement in \cite[Lemma 2.13]{maggidephilippisyoung}. The proof is nearly identical; since the integrands considered in \cite{maggidephilippisyoung} are $C^{2,1}$ rather than $C^{2,\dini}$, it only requires a couple minor technical modifications that we describe below.

\begin{lemma}[Strong maximum principle for $\bF$-minimizers]\label{lemma:strong max principle}
    If $F$ is a uniformly elliptic integrand of class $C^{2,\dini}$, $T$ is an integer rectifiable $m$-current in $\R^{m+1}$ that minimizes $\mathbf{F}$ in $\bB_1$, $\partial T \res{\bB_1}=0$, $v\in \mathbb{S}^m$, $\spt T \cap \bB_1 \subset \{x\cdot v \geq 0\} \cap \bB_1$, and $0 \in \spt T$, then there is $q\in \mathbb{Z}$ and $T'$ with $0\notin \spt T'$ such that, setting $H = \a{\{x\cdot v>0\}}$ oriented by $e_1 \wedge \cdots \wedge e_{m+1}$, 
    $$
    T\res{\bB_1}= q \partial H \res{\bB_1} +T'\,.
    $$
\end{lemma}
\begin{proof}
By a change of coordinates, it suffices to prove the lemma when $v=e_{m+1}$. In a slight abuse of notation, we call $H = \{x_{m+1}>0\}$. We decompose $T\res{\bB_1}=\sum_i \partial \a{E_i}$, where each $E_i$ locally minimizes the anisotropic perimeter in $\bB_1$. We normalize each $E_i$ so that $\partial E_i \cap \bB_1 = \overline{\partial^* E_i} \cap \bB_1$. By interior regularity, i.e., $\mathcal{H}^{m-2}({\rm Sing}_i\, \partial \a{E_i})=0$ \cite{figalli2017regularity}, we have $\mathcal{H}^m(\bB_1 \cap \partial E_i \setminus \partial^* E_i) = 0$ and may assume that $E_i=E_i^{(1)}$ is open, cf. \cite[Lemma 6.2]{de2020uniqueness}. Note that by lower volume density estimates {(cf. \cite[Lemma 2.8]{maggidephilippisyoung}, where the argument does not require any regularity on $D^2 F$, and so applies here)}, only finitely many $E_i$ can have boundary in $\bB_{1/2}$. Thus, to prove the lemma, it suffices to show that for each $E_i$,
\begin{equation*}
    \mbox{either $0 \notin \partial E_i$ or $\partial H \cap \bB_1 \subset \partial E_i$}\,.
\end{equation*}

Fix $i$. We assume that $\partial E_i$ does not contain $\partial H \cap \bB_1$ and aim to prove 
\begin{equation*}
0\notin \partial E_i\,. 
\end{equation*}
Depending on the orientation of $\a{E_i}$ and since $\spt\, T\cap\bB_1\subset H$, we have either $E_i\subset H $ or $E_i^c \subset H$. We set $E=E_i \cap \bB_1$ and $G(x)=F(x)$ if $E_i \subset H$ and $E=E_i^c \cap \bB_1$ and $G(x)=F(-x)$ if $E_i^c \subset H$. We need to show that
\begin{equation}\label{eq:0 not in reflection}
    \mbox{if $E$ minimizes $\int G(\nu_E)$ in $\bB_1$ and $\partial E_i$ does not contain $\partial H \cap \bB_1$, then  $0 \notin \partial E$}\,.
\end{equation}
Equation \eqref{eq:0 not in reflection} is now exactly what is proved in \cite[Lemma 2.13]{maggidephilippisyoung} for $C^{2,1}$ integrands. With this reduction in place, there is only one difference compared to \cite[Lemma 2.13]{maggidephilippisyoung}: their use of $F\in C^{2,1}$ to use classical elliptic regularity theory/maximum principles to derive $u(0)>0$ (\cite[Eq. (2.100)]{maggidephilippisyoung}). Here, we can directly apply our Lemma \ref{lemma:strong max principle for graphs} for $G$ of class $C^{2,\dini}$.

Once the above is settled, the proof in \cite[Lemma 2.13]{maggidephilippisyoung} is finished by an application of \cite[Lemma 2.12]{maggidephilippisyoung} which is our Lemma \ref{lemma:second comp princ}.
\end{proof}

Next, we formulate the barrier principle. 

\begin{proposition}[Anisotropic mean convex barriers]\label{corollary:mean convex barriers}
    Let $F$ be a uniformly elliptic integrand of class $C^{2,\dini}$. If $U\subset \mathbb{R}^{m+1}$ is a {$C^2$} bounded open set such that the anisotropic mean curvatures $H_{\partial U}^F,H_{\partial U}^G\in L^\infty(\partial U)$ corresponding to the anisotropies $\nu\mapsto F(\nu)$ and $\nu \mapsto F(-\nu)=:G(\nu)$ satisfy
    \begin{equation}\notag
    \min_{\partial U}H_{\partial U}^F\geq \kappa\,,\qquad \min_{\partial U}H_{\partial U}^G\geq \kappa    
    \end{equation}
    for some $\kappa>0$, then there is $\delta_0=\delta_0(U)>0$ with the following property:

    \smallskip

    \noindent if $S$ is an $\bF$-minimizer in $\mathbb{R}^{m+1}$ with $\spt(\partial S) \subset \overline{U}$ and $\spt S \subset \overline{B_{\delta_0}(U)}$, then $\spt S \subset \overline{U}$.
\end{proposition}

\begin{remark}[On the assumption $\spt S \subset \overline{B_{\delta_0}(U)}$]
    Without this assumption the result is false, as is seen by the example where $S$ is an equatorial disk in $\mathbb{R}^3$ with boundary $\mathbb{S}^1\times\{0\}$ and $U$ is a small tubular neighborhood containing $\spt \partial S$ but not $\spt S$.
\end{remark}

\begin{proof}
   Let $U_\delta$ denote the bounded open sets $\{ x: {\rm dist}\, (x,U)< \delta\}$. We choose $\delta_0$ small enough so that for all $0\leq \delta \leq \delta_0$,
  \begin{equation}\notag
  \min_{\partial U_\delta}H_{\partial U_\delta}^F\geq \kappa/2\,,\qquad \min_{\partial U_\delta}H_{\partial U_\delta}^G\geq \kappa/2.
    \end{equation}

Now let $S$ be an $\bF$-minimizer as in the statement. Assume for contradiction that $\spt S \setminus \overline{U} \neq \varnothing$. Then there is $\delta_1\geq 0$ which is the smallest $\delta\in [0,\delta_0]$ such that $\spt S \subset \overline{U_{\delta}}$. We must show $\delta_1$ cannot be positive. Since $\spt S$ is compact, there must be $x\in \partial U_{\delta_1} \cap \spt S$. We blow up $S$ at $x$ by some $r_j\to 0$ such that $S_{x,r_j} \to S_0$ for some $S_0$, recall that $S_0$ exists thanks to the mass ratio upper bound in Lemma \ref{lemma:local upper bf bounds}. Note that the blowup satisfies the assumptions of Lemma \ref{lemma:strong max principle}. In particular, $S_0$ decomposes as $q$ copies of an oriented hyperplane plus some current $S'$ such that ${\rm dist}\,(\spt S' \cap \bB_1,0)>0$. If the orientation of $S_0$ at zero coincides with $T_{x}(\partial U_{\delta_1})$ oriented by the outward normal $\nu_{U_{\delta_1}}$ to $\partial U_{\delta_1}$, we define $S'=S$, where if the orientation is opposite, we set $S'=-S$. Note in this latter case that $S'$ is $\mathbf{G}$-minimizer. 

We decompose $S'_{x,r_j}\res \bB_1$ into finitely many boundaries $\sum_i \partial \a{E_i^j}$ of some anisotropic perimeter minimizers in $\bB_1$. Let us assume without loss of generality that $0\in \partial E_1^j$ and $\partial E_1^j$ blows up to a plane for all $j$.  By applying the interior regularity of \cite{figalli2017regularity}, we find that for large $j$, $\partial E_1^j$ is locally a graph of some $f_j\in C^{1,\beta}, \beta >0,$ over $T_x(\partial U_{\delta_1})$, as is $(\partial U_{\delta_1}-x)/r_j$ with some $g_j$. Up to a rotation, let us assume that the tangent plane is $\mathbb{R}^m\times\{0\}$ and that $E_1^j$ and $(U_{\delta_1}-x)/r_j$ correspond to the epigraphs of $f_j$ and $g_j$, so that $g_j \geq f_j$ with equality at $0$. Then $f_j$ and $g_j$ are $C^{1,\beta}$ weak solutions to
\begin{equation}\notag
-{\rm div}\, \left[\nabla F_\S(\nabla f_j)  \right] =0 \,,\qquad -{\rm div}\, \left[\nabla F_\S(\nabla g_j)  \right] \geq r_j\kappa/2\,,
\end{equation}
in the case $S'=S$, with $F_\S$ replaced by $G_\S$ if $S'=-S$. By Lemma \ref{lemma:strong max principle for graphs}, $g_j>f_j$ on a ball containing $0$ for all large $j$, which contradicts $g_j(0)=f_j(0)$. Thus $\delta_1>0$ is impossible and proof is complete.\end{proof}

\subsection{Parametric Hopf lemma}

The Hopf lemma for graphs can be used to prove an analogous statement for an anisotropic area-minimizing hypercurrents sitting above a hyperplane by ``squeezing'' a minimizing graph in between them, an idea which is due to De Philippis and Maggi \cite[Lemma 2.13]{maggidephilippisyoung}.

{\begin{proposition}[Hopf-type lemma for $\bF$-minimizers]\label{lemma:hopf currents}
    If $F$ is a uniformly elliptic integrand of class $C^{2,\dini}$, {$\alpha \in (0,1)$ is as in Proposition \ref{lemma: small data existence}}, $\Omega\subset \mathbb{R}^m$ is an open, star-shaped, bounded set with {smooth} boundary, and $r_0\in (0,1/2)$ is fixed, then there is $\eta_0>0$ with the following property:

    \setlength{\leftskip}{15pt}
    \noindent if $L:\R^m \to \mathbb{R}$ given by $L(x^\prime) = a\cdot x'$ is a linear function with $|a|\leq \eta_0$, $T$ is an $\bF$-minimizer in $\Omega \times (-1,1)$ with $\partial T \res{(\Omega \times (-1,1))}=0$ and 
    \begin{equation}\label{eq:epigraph containment}
    \spt \, T\subset \{(x',x_{m+1}) \in \overline{\Omega}\times [-1,1] : x_{m+1} \geq L(x') \}\,,
    \end{equation}
    and there exists $x_0=(x_0',x_{m+1}^0)\in {\rm graph}_{\partial \Omega} L$ such that
    \begin{equation}\label{eq:xnought rnought away from T}
    {\rm dist}(x_0,\spt T) \geq r_0\,,    
    \end{equation}
    then there are $\tau_0>0$ and $t_0>0$, both depending on $\bf F$, $m$, and $r_0$, and $u\in {\rm Lip}(\Omega)$ such that
    \begin{align}\label{eq:u geq L}
    & u(x') \geq L(x')\qquad \forall x'\in \overline{\Omega}\,,\\ \label{eq:u=L away from x_0'}
    &u(x') = L(x') \qquad \forall x' \in \partial \Omega \setminus B_{r_0}(x_0')\,,\\ \label{eq:hopf for u}
        &u(x'+t\nu_\Omega) >u(x') + ta\cdot \nu_\Omega + |t|\tau_0 \qquad {\forall x' \in \partial \Omega \setminus B_{r_0}(x_0')},\, t\in [-t_0,0)\,,\qquad \mbox{and}\\ \label{eq:doesnt hit subgraph}
        &\spt \, T  \subset \{ (x',x_{m+1})\in \overline{\Omega}\times [-1,1]: x_{m+1} \geq u(x')\}\,.
    \end{align}
    \setlength{\leftskip}{0pt}
\end{proposition}}
\begin{proof}
We choose $\eta_0=\kappa_0/2$, where $\kappa_0$ is as in Proposition \ref{lemma:  small data existence}. Let us fix $r_0 \in (0,1/2)$.

The subsequent argument is divided into two main steps. In the first step, we prove the desired conclusions with additional $L$ and $x_0'$ dependence. More precisely, we fix a linear function $L$ with gradient modulus at most $\eta_0$ and $(x_0',x_{m+1}^0) \in {\rm \graph}_{\partial \Omega}L$, and prove the existence of $\tau_{L,x_0'}$, $t_{L,x_0'}$, and $u_{L,x_0'}$ such that \eqref{eq:u geq L}-\eqref{eq:doesnt hit subgraph} hold for any $T$ as in the assumptions of the lemma. Then in the second step, we prove that $\tau_{L,x_0'}$ and $t_{L,x_0'}$ can be made uniform among $x_0'$ and $L$.
    
    For later use, we let $\varphi \in C^\infty(B_{r_0/2};[0,\infty))$ be such that $\varphi=0$ on $\mathbb{R}^m \setminus B_{r_0/2}$, $0<\varphi<r_0/4$ on $B_{r_0/2}$, and, for every $x' \in \partial \Omega$ and $L(x^\prime) = a\cdot x^\prime$ with $|a|\leq \eta_0$,
    \begin{equation}\label{eq:varphi estimate}
        {\rm dist}\big((x_0^\prime,L(x_0^\prime)+\varphi(0)),\spt \, T\big)>r_0/4\quad \mbox{and}\quad\|L(\cdot)+ \varphi (\cdot - x') \|_{C^{1,\eta}(\Omega)}\leq 3\kappa_0/4\,;
    \end{equation}
such a choice of $\varphi$ is possible since $\eta_0=\kappa_0/2$.
    \medskip
    
    \noindent{\it Step one: \eqref{eq:u geq L}-\eqref{eq:doesnt hit subgraph} hold with $L$ and $x_0'$ dependence.} Fix $L$ and $x_0$ as in the statement of the lemma. Let $G$ be the anisotropy characterized by $\nu \mapsto F(-\nu)=:G(\nu)$. By Proposition \ref{lemma:  small data existence} and \eqref{eq:varphi estimate}, let $u_{L,x_0',F},u_{L,x_0',G}\in C^{1,\alpha}(\Omega)$ be the unique minimizers of $\sF$ and $\mathscr{G}$, respectively (see \eqref{eq:calF def} for the definition of $\sF$), on $\Omega$ with boundary data 
    $$
    \varphi(\cdot-x_0') + L(\cdot).
    $$ 
By Proposition \ref{lemma:  small data existence} and the second estimate in \eqref{eq:varphi estimate}, we have
    \begin{equation}\label{eq:uLxnoughtFG estimates}
        \max\{\|u_{L,x_0',F}\|_{C^{1,\alpha}(\Omega)},\|u_{L,x_0',G}\|_{C^{1,\alpha}(\Omega)}\}\leq C3\kappa_0/4 \,.
    \end{equation}
Setting $u_{L,x_0'}=\min\{u_{L,x_0',F},u_{L,x_0',G}\}$, we point out that \eqref{eq:u geq L} follows by the maximum principle. {Also, \eqref{eq:u=L away from x_0'} holds since $\spt \varphi \subset B_{r_0/2}$ implies that $\varphi(x'-x_0') + L(x') = L(x') $ if $x' \in \partial \Omega \setminus B_{r_0}(x_0')$.}
    
    Next, notice that by the definitions of $r_0$ and the choice $\varphi$ (its support is contained in $B_{r_0/2}$), we have
\begin{equation}\label{eq:epigraph containment in proof}
    \spt \, T\subset \{(x',x_{m+1}) \in \overline{\Omega}\times [-1,1] : x_{m+1} \geq L(x') +\varphi(x'-x_0') \}\,.
\end{equation}
We define
\begin{equation*}
   {\rm M}_F(x) = \int_0^1 D^2 F_\S((1-t)a + t \nabla u_{L,x_0',F}(x))\, dt\quad \mbox{and}\quad {\rm M}_G(x) = \int_0^1 D^2 G_\S((1-t)a + t \nabla u_{L,x_0',G}(x))\, dt\,.
\end{equation*}
Then $u_{L,x_0',F}-L$ satisfies the equation $-\dive [{\rm M}_F\nabla (u_{L,x_0',F}-L)]=0$, which is a uniformly elliptic equation with Dini coefficients (the uniform ellipticity is due to the $L^\infty$ gradient bounds on $u_{L,x_0',F}$ from Proposition \ref{lemma:  small data existence}). Also $u_{L,x_0',G}$ satisfies the analogous equation with $G$ instead of $F$. By the Hopf lemma \cite[Theorem 2.1]{apushkinskaya2019boundary}, the compactness of $\overline \Omega$, and the H\"{o}lder continuity of $\nabla u_{L,x_0',F}$ and $\nabla u_{L,x_0',G}$, there are $\tau_{L,x_0'}$ and $t_{L,x_0'}$ such that 
    \begin{align}\label{eq:first good hopf estimate}
        &u_{L,x_0',F} (x' + t\nu_\Omega)-u_{L,x_0',F}(x') -t L(\nu_\Omega) > |t|\tau_{L,x_0'} \quad \forall x'\in {\partial \Omega \setminus B_{r_0}(x_0')},\, t\in [-t_{L,x_0'},0)\quad \mbox{and}\\ \label{eq:second good hopf estimate}
        &u_{L,x_0',G} (x' + t\nu_\Omega)-u_{L,x_0',G}(x') - tL(\nu_\Omega) > |t|\tau_{L,x_0'} \quad \forall x'\in {\partial \Omega \setminus B_{r_0}(x_0')},\, t\in [-t_{L,x_0'},0)\,.
    \end{align}
Recalling that  $u_{L,x_0'}=\min\{u_{L,x_0',F},u_{L,x_0',G}\}$, we thus obtain \eqref{eq:hopf for u} with $L$ and $x_0'$ dependence. To finish this step, it remains to prove \eqref{eq:doesnt hit subgraph} with $L$ and $x_0'$ dependence.

    Since $\Omega$ (and thus $\Omega \times (-1,1)$) is star-shaped, we can write $T$ as a sum of boundaries of anisotropic relative perimeter minimizers (that is, of $P_F$)
    $$T \res{(\Omega \times (-1,1))}=\sum \a{\partial E_i \cap (\Omega \times (-1,1))}\,.$$
By \eqref{eq:epigraph containment in proof}, for each $E_i$, we have either 
\begin{equation}\label{eq:boundary ordering wmp rdp 2}
    [(\Omega \times (-1,1))\setminus E_i]^{\half} \cap \partial [\Omega \times (-1,1)]\supset \{(x',x_{m+1})\in \partial \Omega\times (-1,1): x_{m+1}<\varphi(x'-x_0')+L(x')\}\cup {(\Omega \times \{-1\})}
\end{equation}
or
\begin{equation}\label{eq:boundary ordering wmp rdp 3}
E_i^{\half} \cap \partial [\Omega \times (-1,1)]\supset \{(x',x_{m+1})\in \partial \Omega\times (-1,1): x_{m+1}< \varphi(x'-x_0')+L(x')\}\cup (\Omega \times \{-1\})\,. 
\end{equation}
    For those $E_i$ such that \eqref{eq:boundary ordering wmp rdp 2} holds, we apply Lemma \ref{lemma:second comp princ} with $u_{L,x_0',G}$, $(\Omega \times (-1,1))\setminus E_i$, and anisotropy $G$ to conclude that
\begin{equation}\label{eq:unique min sits below E 3}
    \big|\{(x',x_{m+1}):x'\in \Omega,\, x_{m+1}\leq u_{L,x_0',G}(x')\} \cap E_i\big|=0\,,
\end{equation}
    whereas for those $E_i$ such that \eqref{eq:boundary ordering wmp rdp 3} holds, we apply Lemma \ref{lemma:second comp princ} with $u_{L,x_0',F}$, $E_i$, and  anisotropy $F$ to conclude that
\begin{equation}\label{eq:unique min sits below E 2}
    \big|\{(x',x_{m+1}):x'\in \Omega,\, x_{m+1}\leq u_{L,x_0',F}(x')\} \setminus E_i\big|=0\,.
\end{equation}

As a consequence of \eqref{eq:unique min sits below E 3} and \eqref{eq:unique min sits below E 2}, if $\spt T$ touches the graph of $u$ inside $\Omega \times \mathbb{R}$, it can only do so when one of these $\partial E_i$'s touches its corresponding graph but does not cross it. Thus to finish proving \eqref{eq:doesnt hit subgraph}, we only need to preclude this possibility. This is impossible because if there were such a touching by some $E_i$, then the blowup of $\partial E_i$ at that point would sit on one side of a plane, and therefore coincide with that plane by Lemma \ref{lemma:strong max principle}. By the interior regularity of \cite{figalli2017regularity} and Lemma \ref{lemma:strong max principle for graphs}, $\spt \partial E_i$ coincides with either the graph of $u_{L,x_0',F}$ or $u_{L,x_0',G}$ in a neighborhood of that touching point, depending on the orientation of $\partial E_i$. This shows that the set of such coincidence points is open relative to the graph, since it is also closed and $\Omega \times \mathbb{R}$ is connected, we conclude that $\spt T$ contains either the graph of $u_{L,x_0',F}$ or $u_{L,x_0',G}$. We obtain a contradiction by recalling \eqref{eq:varphi estimate} and that these functions have $\varphi(x'-x_0') + L(x')$ as boundary data.

\medskip

\noindent{\it Step two: Removing the $L$ and $x_0'$ dependence.} This is a compactness argument. Assume for contradiction that for some sequence $\{(L_j,x_j')\}_j$ satisfying $\|\nabla L_j\|_{L^\infty}\leq \eta_0$ and $x_j\in  {\rm graph}_{\partial\Omega}L_j$, at least one of the sequences of constants $\tau_{L_j,x_j'}$ and $t_{L_j,x_j'}$ converge to $0$ no matter the choices of them for each $j$. Recalling that $u_{L_j,x_j'}=\min\{u_{L_j,x_j',F},u_{L_j,x_j',G}\}$, up to a subsequence, we can use \eqref{eq:uLxnoughtFG estimates} and the Arzela-Ascoli theorem to obtain $x_0'\in \partial \Omega$, $L_0$ with $\|\nabla L_0\|_{L^\infty}\leq \eta_0$, $u_{x_0',L_0,F}$, and $u_{x_0',L_0,G}$ such that 
\begin{equation}\notag
  x_j' \to x_0'\,, \quad \nabla L_j \to \nabla L_0\,, \quad   u_{x_j',L_j,F}\to  u_{x_0',L_0,F} \mbox{ and } u_{x_j',L_j,G} \to u_{x_0',L_0,G} \mbox{ in $C^{1,\beta}$ for all $0<\beta<\alpha$.}
\end{equation}
Furthermore, $u_{x_0',L_0,F}$ and $u_{x_0',L_0,G}$ minimize $\sF$ and $\mathscr{G}$, respectively, for the boundary data $L_0 + \varphi(\cdot - x_0')$. Therefore, the same argument as leading to \eqref{eq:first good hopf estimate}-\eqref{eq:second good hopf estimate} yields
$\tau_{L_0,x_0'}$ and $t_{L_0,x_0'}$ such that, for any $t\in [-t_{L_0,x_0'},0)$,
    \begin{align}\label{eq:first good hopf estimate 2}
        &u_{L_0,x_0',F} (x' + t\nu_\Omega)-u_{L_0,x_0',F}(x') -t L_0(\nu_\Omega) > |t|\tau_{L_0,x_0'} ,\ \forall x'\in {\partial \Omega \setminus B_{r_0}(x_0')},\, \ \mbox{and}\\ \label{eq:second good hopf estimate 2}
        &u_{L_0,x_0',G} (x' + t\nu_\Omega)-u_{L_0,x_0',G}(x') - tL_0(\nu_\Omega) > |t|\tau_{L_0,x_0'}, \ \forall x'\in {\partial \Omega \setminus B_{r_0}(x_0')}\, .
    \end{align}
Since we have $\beta$-H\"{o}lder convergence of gradients and $\nabla L_j \to \nabla L_0$ uniformly, \eqref{eq:first good hopf estimate 2}-\eqref{eq:second good hopf estimate 2} imply that, for $j$ large enough and $t\in [-t_{L_0,x_0'}/2,0)$,
\begin{align*}
        &u_{L_j,x_j',F} (x' + t\nu_\Omega)-u_{L_j,x_j',F}(x') -t L_j(\nu_\Omega) > |t|\tau_{L_0,x_0'}/2, \ \forall x'\in \partial \Omega \cap \{u_{L_j,x_j',F}=L_j\},\, \ \mbox{and}\\ \label{eq:second good hopf estimate 2}
        &u_{L_j,x_j',G} (x' + t\nu_\Omega)-u_{L_j,x_j',G}(x') - tL_j(\nu_\Omega) > |t|\tau_{L_0,x_0'}/2, \ \forall x'\in \partial \Omega \cap \{u_{L_j,x_j',G}=L_j\}\,.
    \end{align*}
But this is a contradiction of our assumption that it was impossible to choose $\tau_{L_j,x_j'}$ and $t_{L_j,x_j'}$ so that neither limit is zero. 
\end{proof}

\section{Boundedness of the polar angle function}

The goal of this section is to bound the polar angle function on $T$.

\subsection{Setup and notation} Throughout this section, we will always be working under the following conditions.

\begin{assumption}\label{assumption:tmx}
    We assume that $F$ is uniformly elliptic of class $C^{2,\dini}$ and $T$ is a nontrivial integer multiplicity rectifiable $m$-current that minimizes $\mathbf{F}$ in $\mathbb{R}^{m+1}$ with boundary $\a{\mathbb{R}^{m-1}\times\{(0,0)\}}$.
\end{assumption}

\begin{definition}[Rays, half hyperplanes, and halfspaces]
    Given $s\in \mathbb{S}^1$, we let $R_s\subset \mathbb{R}^2$ be the ray through the origin containing $s\in \mathbb{S}^1$ and $M_s = \mathbb{R}^{m-1}\times R_s$ be the corresponding half hyperplane. We consider its oriented counterpart
    \begin{equation}\notag
      \a{M_s} :=  \a{\mathbb{R}^{m-1}\times R_s}\,,
    \end{equation}
where the orientation is chosen so that $\partial \a{M_s} = \a{\mathbb{R}^{m-1}\times\{(0,0)\}}$. We will also denote by $H_s$ the open halfspace
\begin{equation}\notag
   \mathbb{R}^{m-1}\times \{r\sigma : r>0,\, \sigma \in \mathbb{S}^1,\, {\rm dist}_{\mathbb{S}^1}(\sigma,s)<\pi/2 \}\,, 
\end{equation}
so that $M_s$ is perpendicular to $\partial H_s$ and $M_s \subset H_s$. To streamline notation, we will often view $\mathbb{S}^1$ as a subset of $\mathbb{C}$ rather than $\mathbb{R}^2$, so that we can describe its elements by $e^{it}$ for $t\in \mathbb{R}$.
\end{definition}

\subsection{The polar angle function and its super-/sublevel sets}

In this subsection we define the polar angle function $\boldsymbol{\theta}$ following \cite[Section 11]{HS}.

First, we need some notation. Following \cite[Section 11]{HS}, we define $\mathbf{q}:\mathbb{R}^{m+1}\to \mathbb{C}$ by
\begin{equation}\notag
   \mathbf{q}(x_1,\dots,x_{m+1})=x_m + ix_{m+1}\,.
\end{equation}
For any piecewise $C^1$ curve $\gamma:[0,1]\to \mathbb{R}^{m+1} \setminus \mathbf{q}^{-1}(0)$, we set
\begin{equation}\notag
    \theta(\gamma) = {\rm Im}\left( \int_{\mathbf{q}\circ \gamma}\frac{dz}{z}\right)\,.
\end{equation}

The statement and proof of the following lemma are the same as in \cite[Section 11]{HS}, except there the curves in question are confined to the link of the blowup cone; we include it for convenience. 

\begin{lemma}\label{lemma:theta well defined}
If $T$ satisfies Assumption \ref{assumption:tmx} and $\gamma:[0,1]\to {\rm Reg}_i\, T$ is piecewise $C^1$ with $\gamma(0)=\gamma(1)$, then $\theta(\gamma)=0$.
\end{lemma}

\begin{proof}
    First, we note that for $\rho>0$ small enough, the modified curve $\beta:=\gamma + \rho \nu_{{\rm Reg}_i\,T}\circ \gamma$
    satisfies 
    \begin{equation}\label{eq:beta doesnt hit support T}
        \theta(\gamma)=\theta(\beta)\,,\qquad \beta([0,1]) \cap \spt \, T = \varnothing\,,
    \end{equation}
where $\theta(\gamma)=\theta(\beta)$ follows from the homotopy invariance of winding number. As observed in \cite[Section 11]{HS}, $\theta(\beta)$ corresponds to algebraically counting the number of crossings (modulo a multiplicative factor of $2\pi$) of $\gamma$ with any oriented half-hyperplane $P=(-1)^{m-1}\a{\mathbb{R}^{m-1}\times L}$, where $L\subset \mathbb{R}^2$ is any ray chosen so that $\beta$ intersects $\spt(P)$ transversally; thus
   \begin{equation}\label{eq:gamma=beta}
       \theta(\gamma)=\theta(\beta)=2\pi(P \cap \beta_\#[0,1])(1)\,,
   \end{equation}
   where the current in the right-most term is the \emph{intersection between two currents} as defined in \cite[4.3.20, page 460]{Fed}.
   Note that since $\beta([0,1])$ is compact, $1$ is an admissible test function for the current $P \cap \beta_\#[0,1]$. 
   
   Next, fix a cube $[-R,R]^{m+1}$ large enough so that $\beta([0,1]) \subset\!\subset [-R,R]^{m+1}$. Since $\partial P - \partial T=0$, we can find an integer rectifiable current $J$ on $\mathbb{R}^{m+1}$ such that
   \begin{equation}\label{eq:H-T is a boundary}
       (P-T)\res{(-R,R)^{m+1}} = (\partial J)\res{(-R,R)^{m+1}}\,.
   \end{equation}
Using in order \eqref{eq:gamma=beta}, \eqref{eq:H-T is a boundary}, and the second assertion in \eqref{eq:beta doesnt hit support T}, we compute
   \begin{equation*}
       (2\pi)^{-1}\theta(\gamma) = \left([T + \partial J]\cap \beta_\#[0,1] \right)(1)= 0 + \left( \partial J\cap \beta_\#[0,1] \right)(1)\,.
   \end{equation*}
To simplify the right hand side we will use the following formula from \cite[bottom of page 460]{Fed} for the boundary of the intersection current applied to $J\cap\beta_\#[0,1]$:
   \begin{equation*}
       \partial (J\cap\beta_\#[0,1]) - J \cap (\partial \beta_\#[0,1]) = (-1)^m (\partial J)\cap \beta_\#[0,1].
   \end{equation*}
Recall that $\partial S(1) = S(d 1 ) =0$ for any compactly supported $1$-current $S$. We apply this formula with $S = J\cap\beta_\#[0,1]$, use that $\partial \beta_\#[0,1]=0$ (since $\beta$ is closed), and put together the last two displayed equations to obtain
    \[ (2\pi)^{-1}\theta(\gamma) =(-1)^m \partial (J\cap\beta_\#[0,1])(1) - (-1)^mJ \cap (\partial \beta_\#[0,1])(1)=0+0=0.\qedhere\]
\end{proof}

The preceding lemma ensures that the following definition is well-posed and depends only on the initial choice of base points $a_j$; cf. \cite[page 478]{HS}.

\begin{definition}[Polar angle function {on $\bB_{10}\cap {\rm Reg}_i T$ and its super-/sublevel sets}]\label{def:polar angle}
    Suppose $T$ satisfies Assumption \ref{assumption:tmx} and $\{C_j\}$ are the path connected components of ${\bB_{10} \cap {\rm Reg}_i\,T}$. For each $C_j$, we choose base point $a_j\in C_j$ and $\theta_j \in [0,2\pi)$ such that $\theta_j$ is the polar angle of $\mathbf{q}(a_j)$. Then we define the continuous function $\boldsymbol{\theta}:{\bB_{10}\cap }{\rm Reg}_i\, T\to \mathbb{R}$ by
\begin{equation*}
    \boldsymbol{\theta}(x) = \theta_j + \theta(\gamma)
\end{equation*}
if $x\in C_j$ and $\gamma:[0,1]\to C_j$ is any $C^1$ curve with $\gamma(0)=a_j$ and $\gamma(1)=x$. For $N\in \mathbb{N}$, we set
\begin{equation*}
\begin{aligned}
    T_{N}:= \a{\{x\in \bB_{10} \cap {\rm Reg}_i T : \boldsymbol{\theta}(x) >  2N \pi\}}\,, \\
    T_{-N}:= \a{\{x\in \bB_{10} \cap {\rm Reg}_i T : \boldsymbol{\theta}(x) <  -2N \pi\}}\,.    
\end{aligned}
\end{equation*}
\end{definition}

\begin{lemma}[Super-/sublevel sets]\label{lemma: basic prop T N}
Let $F$ and $T$ be as in Assumption \ref{assumption:tmx}. Then for all $N \in \mathbb{N}$:
\begin{enumerate}[\upshape (i)]
     \item $T_{N}$ and $T_{-N}$ are integer rectifiable $m$-currents, and

    \item $\partial T_{\pm N}$ is supported in $\overline{M_1}$.
\end{enumerate}
\end{lemma}
\begin{proof}
Item (i) holds since the \emph{sets} defining $T_{\pm N}$ are $\cH^m$-measurable sets contained in a $C^1$ submanifold of dimension $m$, namely, ${\rm Reg}_i\, T$, and are therefore locally $m$-rectifiable. For item (ii), fix $x\in \spt\, T_{\pm N} \setminus \overline{M_1}$. Then we can choose $r>0$ small enough so that ${\rm dist}(\bB_r(x),M_1)=: d >0$. Since $\{\boldsymbol{\theta}=\pm 2N \pi\}\subset \overline{M_1}$, there is $\delta>0$ depending on $d$ such that $\boldsymbol{\theta}\geq 2N\pi + \delta$ on $\bB_r(x) \cap {\rm Reg}_i T$ if $x\in \spt T_N \setminus \overline{M_1}$ and $\boldsymbol{\theta}(x)\geq 2N\pi $ or $\boldsymbol{\theta}\leq -2N\pi - \delta$ on $\bB_r(x) \cap {\rm Reg}_i T$ if $x\in \spt T_{-N} \setminus \overline{M_1}$ and $\boldsymbol{\theta}(x) \leq -2N\pi $. Thus $T_{\pm N}\res{\bB_r(x)}$ is concentrated on a union of connected components of ${\rm Reg}_i T \cap \bB_r(x)$. Since $T$ is boundaryless on that ball and thus each connected component of ${\rm Reg}_i T$ is, $T_{\pm N}$ is boundaryless there as well and the proof is complete.
\end{proof}

\subsection{Barrier construction}\label{subsec:barrier construction} In this subsection, we build an anisotropic mean convex barrier $U$ to trap the support of $T_{\pm N}\res \bB_5$ inside $\overline{U}$ for large $N$, from which we will conclude that $\boldsymbol{\theta}$ is bounded.

\begin{theorem}[$\boldsymbol{\theta}\in L^\infty(\bB_5)$]\label{lemma:trapping lemma}
If $T$ satisfies Assumption \ref{assumption:tmx}, then there is an anisotropic mean convex smooth open set $U\subset \R^{m+1}$ and $N_0 = N_0(U, {T})>0$ such that for any $N\geq N_0$, 
\begin{equation}\label{eq:barrier containment equation in proposition}
   {\spt(T_{\pm N}\res \bB_{5}) \subset \overline{U \cap \bB_{5}}\subset \overline{H_1}}.
\end{equation}
As a consequence,
\begin{equation*}
    \boldsymbol{\theta}\in L^\infty(\bB_5)\,.
\end{equation*}
\end{theorem}
\begin{proof}
The proof is divided into steps. We prove it for $T_N$ only since the argument for $T_{-N}$ is entirely analogous. 

\medskip

In Step 1, we will show that the supports of $T_N$ and $T_{-N}$ are trapped in a small neighborhood of the boundary $\R^{m-1}\times \{0\}^2$. More precisely, we show that for any {$0<\delta<1/2$}, there is $N_0\in \mathbb{N}$ such that
\begin{equation}\label{eqn: trapping in small nghd}
    \spt(T_N \cup T_{-N}) \cap \overline{\bB_{8}} \subset B_{\delta}(\R^{m-1}\times \{0\}^2) \qquad \forall N \geq N_0.
\end{equation}
In Step 2, we construct an anisotropic mean convex smooth open set $U$ such that $U\cap \bB_{{5}} \subset H_1$ and, for any $N>N_0$, we use Step 1 to show 
\[\spt(\partial (T_N\res\bB_{5})) \subset \overline{U}\text{ and }\spt(T_N\res\bB_{5}) \subset \overline{B_\delta(U)}\,.\]
For $\delta$ small enough depending on $U$ and $N>N_0$, the last displayed equation allows us to apply Proposition \ref{corollary:mean convex barriers} to $S = T_N\res \bB_{5}$, which yields $\spt( T_N\res \bB_{5})\subset \overline{U}$. {Since $\overline{U}\cap\bB_5 \subset \overline{H_1}$, this implies that $\boldsymbol{\theta}\leq 2\pi N +\pi/2$ on $\bB_5$.} This will conclude the proof of the theorem.

\noindent{\it Step 1}: It is enough to prove the claim for $T_N$; the proof for $T_{-N}$ is the same. By way of contradiction, assume \eqref{eqn: trapping in small nghd} fails, that is, there is $0<\delta<1/2$ such that for any $N$ large enough there exists
\begin{equation}\label{eqn: trapping fails}
        x_N \in \spt(T_N) \cap \left(\overline{\bB_{8}}\setminus  B_{\delta}(\R^{m-1}\times \{0\}^2)\right).
\end{equation}
Since the set in \eqref{eqn: trapping fails} is contained in the compact set $ \spt(T) \cap \left(\overline{\bB_{8}}\setminus  B_{\delta}(\R^{m-1}\times \{0\}^2)\right)$, then up to a subsequence, there is a limit $x_\infty \in  \spt(T) \cap \left(\overline{\bB_{8}}\setminus  B_{\delta}(\R^{m-1}\times \{0\}^2)\right)$ of the sequence $\{x_N\}_N$. Up to a rotation, we may assume that $x_\infty \in M_1$. Observe that, since $\bB_{\delta/2}(x_\infty)\subset H_1$, we get
$$\boldsymbol{\theta}(\bB_{\delta/2}(x_\infty)) \subset \cup_{z\in \mathbb{Z}}\left(-\pi/2+2\pi z, \pi/2 + 2\pi z\right).$$ 
Letting $S_z = \a{\{ x\in \bB_{\delta/2}(x_\infty)\cap {\rm Reg}_i\, T : \boldsymbol{\theta}(x) \in (-\pi/2+2\pi z, \pi/2+ 2\pi z)\}}$, note that by definition, ${\rm Reg}_i S_z \cap {\rm Reg}_i S_{z'} = \varnothing$ if $z\neq z'$. In particular, $\sum_z \bF(S_z) = \bF(T \res{\bB_{\delta/2}(x_\infty)})$, and each $S_z$ is $\bF$-minimizing on $\bB_{\delta/2}(x_\infty)$. 

{Up to relabelling, let $\{S_{z_j}\}_{j=1}^J$ be those currents such that $\spt S_{z_j} \cap \bB_{\delta/4}(x_\infty)\neq \varnothing$.} We claim that $J$ is finite. Indeed, by the lower density bounds (see e.g.~ \cite[Lemma 4]{bombieri1982regularity}) applied to each $S_{z_j}$, we may estimate
\begin{equation*}
    +\infty > \Lambda |T|(\bB_{\delta/2}(x_\infty)) \geq \bF(T\res \bB_{\delta/2}(x_\infty)) \geq \sum_{j=1}^{J} \bF(S_{z_j};\bB_{\delta/4}(x_\infty)) \geq c_0{\frac{\delta^m}{4^m}}\sum_{j=1}^{J}1\,.
\end{equation*}
This implies that $J<+\infty$, {and as a consequence that $\boldsymbol{\theta} \leq \pi/2 + 2\pi \max_j z_j$ on $\bB_{\delta/4}(x_\infty)$}. 
But this contradicts $x_N \in \spt T_N$ for large enough $N$. 

\noindent{\it Step two}: Fix small $\epsilon \in (0,1)$ to be chosen later. Consider the set
\begin{equation}\notag
  C_\epsilon = \left(B^{m-1}_{{7}}\times B_\epsilon^{2} \right)\cup D_\epsilon ,
\end{equation}
where $B^{m-1}_{7}$ and $B_\epsilon^2$ denote the balls of radius ${7}$ and $\epsilon$ in $\R^{m-1}$ and $\mathbb{R}^2$, respectively, and we choose the ends $D_\epsilon$ so that $C_\epsilon$ is $C^\infty$, {convex}, and the anisotropic mean curvatures $H_{\partial C_\epsilon}^F$ and $H_{\partial C_\epsilon}^G$ (where $G(\nu)=F(-\nu)$) are of order $1/\epsilon$ for all small $\epsilon$. 

Now we compose $C_\epsilon$ with a dilation map that stretches the cylinder {in the $(x_m,x_{m+1})$ directions} for ${|(x_1,\dots,x_{m-1})|>6}$. Specifically, let $\varphi \in C^\infty(\mathbb{R};[0,1])$ be such that $\varphi\equiv 0$ on $[0,{6}]$ and $\varphi=1$ on $[{6.5},\infty)$. For $x = (x^\prime, y) \in \R^{m-1}\times \R^2$, we define a dilation map and the mean convex barrier set as
\begin{equation*}
   d_\epsilon(x) := (x^\prime, (1+\epsilon^2\varphi(|x^\prime|))y) \qquad \text{ and }\qquad S_\epsilon := d_\epsilon(C_\epsilon) + (0_{m-1},\epsilon,0).
\end{equation*}
For all $\epsilon$ small enough, $S_\epsilon$ is anisotropic mean convex and $\bG$-mean convex since $C_\epsilon$ is convex with anisotropic mean curvature order $1/\epsilon$ and $d_\epsilon$ is ${\epsilon^2}$-close to the identity. Choosing such an $\epsilon$, from the definitions, we see that
\begin{equation*}
(\R^{m-1}\times \{0\}^2){\cap} \bB_{{6}}\subset \partial S_\epsilon\,.
\end{equation*}
Furthermore, we claim that due to the dilation (which fattened $S_\epsilon$ near its ``ends''), $S_\epsilon$ satisfies
\begin{equation}\label{eqn: Seps contains bdr points}
    \partial B_{r}^{{m-1}}\times \{(0,0)\} \subset\!\subset S_\epsilon\qquad {\forall r\in (6.5,7)}\,.    
\end{equation} 
Indeed, given $z'\in\partial B^{m-1}_{r}$, let $x:= (z', -\frac{\epsilon}{1+{\epsilon^2}}, 0){\in C_\epsilon}$, {and note that ${\rm dist}(x,\partial C_\epsilon)>0$}. Since ${\epsilon^2\varphi({|z'|})=\epsilon^2}$, we thus have ${(z',0,0)} = d_\epsilon(x) + (0_{m-1}, \epsilon, 0)$ and ${{\rm dist}({(z',0,0)},\partial S_\epsilon)>0}$. {Lastly, before moving on to the barrier argument, we fix $\delta_0(S_\epsilon)$ as in Proposition \ref{corollary:mean convex barriers}.}

Moving on to the barrier argument, by slicing theory, for almost every $r>0$, we know that $\partial (T_N\res\bB_{r}) = \langle T_N, |\cdot|, r\rangle + (\partial T_N)\res \bB_{r}$ for every $N$. {We fix such an $r\in (6.5,7)$.} By Step 1 applied with any $\delta <{\min\{\epsilon^2,{\rm dist}(\partial B_r^{m-1}\times \{0\}^2,\partial S_\epsilon)/2,\delta_0\}}$ and \eqref{eqn: Seps contains bdr points}, it follows that {there is $N_0$ such that for all $N \geq N_0$,}
\begin{equation*}
    \spt(\langle T_N, |\cdot|, r\rangle) \subset \partial \bB_{r} \cap B_\delta (\R^{m-1}\times \{0\}^2) \subset S_\epsilon\,.
\end{equation*}
Since $\spt (\partial T_N)\subset \overline{M_1}{\cap B_{\delta}(\R^{m-1}\times \{0\}^2)}$, according to Lemma \ref{lemma: basic prop T N}, and by Step 1, we obtain 
\begin{equation}\notag
\spt((\partial T_N)\res\bB_{r})\subset \overline{M_1}{\cap B_{\delta}(\R^{m-1}\times \{0\}^2)}\cap \bB_{r}\subset \overline{S_\epsilon}  
\end{equation}
since {$\delta \leq \epsilon^2 \leq 2\epsilon$, where $2\epsilon$ is the ``width'' of $S_\epsilon \cap \bB_r$ in the $(x_m,x_{m+1})$ directions}. We have then proved 
\begin{equation}\label{eq:boundary contained in barrier}
\spt(\partial (T_N\res\bB_{r})) \subset \overline{S_\epsilon}\qquad \mbox{for all $N\geq N_0$}.
\end{equation}
We also notice that {since $(\mathbb{R}^{m-1}\times \{0\}^2) \cap \bB_r \subset \overline{S_\epsilon}$,} the application of Step 1 directly implies that
\begin{equation}\label{eq:distance from barrier}
 \spt(T_N\res\bB_{r})\subset \overline{B_{\delta_0}(S_\epsilon)}\,.    
\end{equation}
Equations \eqref{eq:boundary contained in barrier}-\eqref{eq:distance from barrier} allow for the application of Proposition \ref{corollary:mean convex barriers} to $T_N \res B_r$, yielding 
\begin{equation}\notag
  \spt (T_N \res \bB_r) \subset \overline{S_\epsilon}=: \overline{U}\,.  
\end{equation}
Since $r>5$ and $\overline{S_\epsilon} \cap \bB_5 \subset \overline{H_1}$, this concludes the proof of \eqref{eq:barrier containment equation in proposition}.
\end{proof}

\section{Existence of a flat blowup}

The goal of this section is to prove Theorem \ref{thm:main existence theorem intro}. 

\begin{remark}[Reduction of Theorem \ref{thm:main existence theorem intro} to the flat case]\label{remark:reduction to flat case}
We observe that a simple argument shows that up to a rotation, any blowup of $T$ is a non-trivial $\bF$-minimizer on $\mathbb{R}^{m+1}$ with flat boundary $\a{\mathbb{R}^{m-1}\times \{0\}^2}$. Indeed, if a blowup $T_0$ of $T$ is of mass zero, then we must have $\partial T_0 =0$ contradicting that the blowup limit of the boundary is flat, which follows from the fact that the boundary is differentiable at $0$. Therefore, we can work under Assumption \ref{assumption:tmx}.
\end{remark}

We will need one more definition, which describes a family of cylinders that are rotations of each other around the boundary $\mathbb{R}^{m-1}$ and all of which contains a portion of the boundary.

\begin{definition}[Cylinders tangent to the boundary]\label{def:cylinders}
    Let $C\subset \mathbb{R}^m$ be a smooth, convex, open set such that $\partial C$ is diffeomorphic to $\mathbb{S}^{m-1}$, and
    \begin{align*}
    &\{ x\in \mathbb{R}^m : |x|<1/2, \, x_m>0\}\subset C \subset \{x \in \mathbb{R}^m :|x|<2,\, x_{m}>0\}\,\qquad \mbox{and}\\ \notag
      &B_{1}^{m-1}(0) \times \{0\}\subset \partial C\,.  
    \end{align*}
For $s\in \mathbb{S}^1$ and $t>0$, we define
\begin{equation}\notag
  C_{s,t} = t\left[ R_s\left( C \times (-1,1)  \right)\right]  \,,
\end{equation}
where {$R_s$ is the rotation in $\mathbb{R}^{m+1}$ which fixes $(x_1,\dots,x_{m-1})$ and sends $(x_m, x_{m+1})\sim re^{i\theta}$ to $sre^{i\theta}$.}
\end{definition}

\begin{theorem}[Flat blowups]\label{thm:ex and uni blowup section}
If $T$ satisfies Assumption \ref{assumption:tmx}, then there exists a half-hyperplane $M_s$ and $Q\in \mathbb{N}$ such that a blowup of $T$ at $0$ is given by $Q\a{M_s} - (Q-1)\a{M_{-s}}$.
\end{theorem}

\begin{proof}
   The proof is divided into steps. 

\medskip

   \noindent{\it Set-up.} By Theorem \ref{lemma:trapping lemma}, we know that $\boldsymbol{\theta}$ is bounded on $\bB_5\cap{\rm Reg}_i\, T$. We set
   \begin{equation}\notag
   \boldsymbol{\theta}_r:=\inf \{\boldsymbol{\theta}(x) : x\in \bB_r \cap {\rm Reg}_i T  \}  \qquad \mbox{and}\qquad \boldsymbol{\theta}_{\rm min} = \lim_{r\searrow 0}  \boldsymbol{\theta}_r;
   \end{equation}
   note that the limit exists since $\boldsymbol{\theta}$ is bounded and $\boldsymbol{\theta}_r$ is a decreasing function of $r$. Up to a rotation, we may assume that $\boldsymbol{\theta}_{\rm min}\in 2\pi \mathbb{Z}$. Our candidate tangent currents are thus
\begin{equation}\notag
    Q \a{M_{1}}- (Q-1) \a{M_{-1}}\qquad \mbox{and}\qquad Q \a{M_{-1}}- (Q-1) \a{M_{1}} \qquad \qquad Q \in \mathbb{N}\,,
\end{equation}
that is, those flat cones which contain $ M_{1}=M_{e^{i\boldsymbol{\theta}_{\rm min}}}$ in their supports.

\medskip
   
\noindent{\it Alternative characterization of $\boldsymbol{\theta}_{\rm min}$}. The goal of this step is to obtain $\boldsymbol{\theta}_{\rm min}$ as a limit over nested cylinders tangent to the boundary at the origin rather than nested balls containing the origin. 
   
We begin by observing that for $r$ small enough such that
   \begin{equation}\notag
       \inf \{\boldsymbol{\theta}(x) : x\in \bB_r \cap {\rm Reg}_i T  \}> \boldsymbol{\theta}_{\rm min} - \pi/4 \,,
   \end{equation}
we may assume that a sequence $\{x_j^r\}_j$ with polar angles converging to the infimum on $\bB_r \cap {\rm Reg}_i T$ is chosen such that
\begin{equation}\notag
    \{\boldsymbol{\theta}(x_j^r)\}_j\subset (\boldsymbol{\theta}_{\rm min}-\pi/4,\boldsymbol{\theta}_{\rm min}]\,;
\end{equation}
in other words, for all small $r$, the optimal sequences are contained in the halfspace $H_1$. With this and recalling the definition of the cylinders $C_{s,t}$ (Definition \ref{def:cylinders}), we obtain
\begin{equation}\label{eq:cyl and round min comparison}
\begin{aligned}
       \inf \{\boldsymbol{\theta}(x) : x\in \bB_{2r} \cap {\rm Reg}_i T  \} &\leq  \inf \{\boldsymbol{\theta}(x) : x\in C_{1,r} \cap {\rm Reg}_i T  \}  \\
       &{\leq \inf \{\boldsymbol{\theta}(x) : x\in H_1\cap\bB_{r/2} \cap {\rm Reg}_i T  \}}\\
       &= \inf \{\boldsymbol{\theta}(x) : x\in \bB_{r/2} \cap {\rm Reg}_i T  \} \,.
\end{aligned}
\end{equation}

Since the quantity 
\begin{equation}\notag
   \boldsymbol{\theta}_{\rm cylindrical\,min}:= \lim_{r\searrow 0} \Big[\inf \{\boldsymbol{\theta}(x) : x\in C_{1,r} \cap {\rm Reg}_i T  \}    \Big]
\end{equation}
is also determined by an infimum which is decreasing in $r$, \eqref{eq:cyl and round min comparison} implies that
\begin{equation}\label{eq:cylindrical min equals round min}
    \boldsymbol{\theta}_{\rm cylindrical\,min} = \boldsymbol{\theta}_{\rm min}\in 2\pi \mathbb{Z}\,.
\end{equation}

\medskip

\noindent{\it Every blowup contains $M_{1}$}. Here we prove that for any sequence $r_j\searrow 0$ and blowup $T_0$ such that $T_{0,r_j}\to T_0$, there is $z\in \mathbb{N}$ such that 
\begin{equation}\label{eq:blowup contains theta min plane}
    M_{1} \subset \spt\, T_0 \qquad \mbox{and}\qquad \Theta_{T_0}^m \equiv z \mbox{ on $M_1$}\,.
\end{equation}

First, note that every point  $x\in {\rm Reg}_i T\cap C_{1,1}$ satisfies $\boldsymbol{\theta}(x) \in (-\pi/2,\pi/2) + 2\pi \mathbb{Z}$. Therefore, recalling that $\boldsymbol{\theta}_{\rm min}\in 2\pi \mathbb{Z}$, we can write ${\rm Reg}_i T\cap C_{1,1}=A_1 \cup A_2$, where
\begin{equation}\notag
  A_1 = \{x\in {\rm Reg}_i T\cap C_{1,1} : \boldsymbol{\theta}(x)< \boldsymbol{\theta}_{\rm min}+ \pi/2 \}\,, \qquad   A_2 = \{x\in {\rm Reg}_i T\cap C_{1,1} : \boldsymbol{\theta}(x)> \boldsymbol{\theta}_{\rm min}+ 3\pi/2 \}\,. 
\end{equation}
Define the $\bF$-minimizers $S_k := T_k\res A_k$ for $k=1,2$ whose support is contained in $\overline{C_{1,1}}$. 
Up to an unrelabelled subsequence, let us assume that $(S_1)_{0,r_j}$ converges to some $\bF$-minimizer $S_0$. Since the outward normal vector to $\partial C_{1,1}$ at $0$ is $-e_m$, then the blowup of $C_{1,1}$ at the origin is the halfspace $H_1$. Thus $S_0$ is an $\bF$-minimizer in $H_1$.

Next, recalling that $H_i$ is the open halfspace with boundary $\{x_{m+1}=0\}$, we claim that 
\begin{equation}\label{eq:containment in quadrant}
\spt S_0 \subset \overline{H_1} \cap \overline{H_{i}}\,.    
\end{equation}
The containment in $\overline{H}_1$ is already established above. 
For the containment in $\overline{H}_i$, we observe that, by definition,
\begin{equation}\label{eq:theta def inequality}
   \boldsymbol{\theta}_{\rm min}+\pi/2> \boldsymbol{\theta}(y) \geq \inf \{\boldsymbol{\theta}(x) : x\in C_{1,r} \cap {\rm Reg}_i S_1  \}=: \boldsymbol{\theta}_{r}^{\rm cyl}\geq \boldsymbol{\theta}_{\rm min}- \pi/2, \qquad \forall y\in C_{1,r} \cap {\rm Reg}_i S_1 \,.
\end{equation}
Phrased in terms of halfspaces, \eqref{eq:theta def inequality} says that
\begin{equation}\label{eq:theta def inequality halfspace version}
    C_{1,r} \cap {\rm Reg}_i S_1 \subset C_{1,r} \cap \overline{H_{e^{i(\boldsymbol{\theta}_{r}^{\rm cyl}+\pi/2)}}}\,.
\end{equation}
Recalling from \eqref{eq:cylindrical min equals round min} that $\boldsymbol{\theta}_r^{\rm cyl}\to \boldsymbol{\theta}_{\rm min}\in 2\pi \mathbb{Z}$ as $r\to 0$, taking $r\to 0$ yields
\begin{equation}\notag
   \spt \, S_0 \subset \overline{H_{e^{i(\boldsymbol{\theta}_{\rm min}+\pi/2)}}}=\overline{H_i}\,.
\end{equation}

Now we assume for contradiction that 
\begin{equation}\label{eq:m1 not in the support contra hyp}
    M_1 \not\subset \spt T_0\,.
\end{equation}
Since $\spt S_0\subset \spt T_0$, then there exists $(x_0',0) \in M_1 \setminus \spt S_0$ with
\begin{equation}\label{eq:xnought far from spt snought}
    {\rm dist}((x_0',0), \spt S_0) =: r_0>0.
\end{equation}
Let $r_1$ be such that $(x_0',0)\in \partial C_{1,r_1}$. We claim that for all large $j$, we may apply Proposition \ref{lemma:hopf currents} to $(S_1)_{0,r_j}$ on $C_{1,r_1}$. Indeed, by \eqref{eq:theta def inequality halfspace version}, for $j$ large, we have 
\begin{equation}\label{eq:epigraph containment blowup proof}
    \spt \, \big[(S_1)_{0,r_j}\res{C_{1,r_1}}\big]\subset \{(x',x_{m+1}) \in \overline{C_{1,r_1}} : x_{m+1} \geq L_j(x') \}\,,
    \end{equation}
for some linear functions $L_j$'s which converge uniformly to zero. Moreover, since $\spt (S_1)_{0,r_j}$ converge in the local Hausdorff sense to $\spt S_0$, \eqref{eq:xnought far from spt snought} and the uniform convergence of $L_j$ to $0$ imply that, for large $j$, 
\begin{equation}\label{eq:xnought far from spt S1/rj}
    {\rm dist}((x_0',L_j(x_0')), \spt (S_1)_{0,r_j}) \geq r_0/2\,.
\end{equation}

Together with the convergence of $L_j$ to $0$, \eqref{eq:epigraph containment blowup proof} and \eqref{eq:xnought far from spt S1/rj} allow for the application of Proposition \ref{lemma:hopf currents}. Then the conclusions \eqref{eq:hopf for u}-\eqref{eq:doesnt hit subgraph} of the proposition imply that there is $\theta_0>\pi/2$ and $0<r_2<r_1$ such that 
\begin{equation*}
    \spt (S_1)_{0,r_j} \cap C_{1,r_2} \subset \overline{H_{e^{i\theta_0}}}\,;
\end{equation*}
in words, $\spt (S_1)_{0,r_j} \cap C_{1,r_2}$ is contained in a wedge with boundary coming in to the origin with angle strictly greater than $0$. Using that $\boldsymbol{\theta}_{\rm min}\in 2\pi \mathbb{Z}$, we get existence of $\delta>0$ such that
\begin{equation}\notag
\boldsymbol{\theta}(y)>\boldsymbol{\theta}_{\rm min} + \delta    \qquad \forall y\in \spt S_1 \cap C_{1,r_2r_j}\,.
\end{equation}
By definition of $A_1$ and $S_1$, the last displayed inequality contradicts \eqref{eq:cylindrical min equals round min}. Thus \eqref{eq:m1 not in the support contra hyp} is false, that is $M_1 \subset \spt T_0$. 
By Lemma \ref{lemma:strong max principle} (which applies due to \eqref{eq:containment in quadrant}), we deduce that the multiplicity of $T_0$ is constant on $M_1$, which finishes the proof of \eqref{eq:blowup contains theta min plane}.

\medskip

\noindent{\it Blowups containing $M_1$ in their support have flat blowups.} Here, we prove the following. Let $T_0$ be a blowup of $T$ at $0$ such that $M_1\subset \spt\, T_0$, then any blowup $T_{00}$ of $T_0$ at the origin satisfies 
\begin{equation*}
   \spt T_{00} \subset \partial H_i\,.
\end{equation*}

First, we observe that \eqref{eq:blowup contains theta min plane} implies that $M_1$ is a connected component of ${\rm Reg}_i T_0$, and so no other connected component of ${\rm Reg}_i T_0$ can intersect it by Lemma \ref{lemma:strong max principle} and interior regularity. As a consequence, if we define as in Definition \ref{def:polar angle} a polar angle function $\boldsymbol{\theta}_{T_0}$ for the blowup with domain in ${\rm Reg}_i T_0$, it can never take a value in $2\pi \mathbb{Z}$. Therefore, $\boldsymbol{\theta}_{T_0}$ is globally bounded, and we may as well assume that it takes values off of $M_1$ in the open interval $(0,2\pi)$. 

We claim that there is $\delta\in (0,\pi/2)$ such that 
\begin{equation}\label{eq:restricted theta range}
    \boldsymbol{\theta}_{T_0}(\bB_1 \cap {\rm Reg}_i T_0 \setminus M_1) \subset (\delta, 2\pi - \delta)\,.
\end{equation}
To see this, let $R$ be chosen large enough so that $C_{1,R}$ contains $\bB_1 \cap \{x_m>0\}$. Letting $z$ be as in \eqref{eq:blowup contains theta min plane}, we split $(T_0-z\a{M_1})\res C_{1,R}$ into two pieces: $T_1 := (T_0-z\a{M_1})\res C_{1,R}\cap \{\boldsymbol{\theta}_{T_0} \in (0,\pi/2)\}$, that is, the part of $( T_0- z\a{M_1})\res C_{1,R}$ living above $M_1$, and $T_2$ is the restriction to those points with $\boldsymbol{\theta}_{T_0}\in (3\pi/2,2\pi)$. Since $T_0 - z\a{M_1}$ can never take polar angle value in $2\pi \mathbb{Z}$, the conditions \eqref{eq:epigraph containment} and \eqref{eq:xnought rnought away from T} are satisfied by $T_1$ with $L=0$. Thus \eqref{eq:hopf for u}-\eqref{eq:doesnt hit subgraph} yield $\delta_1>0$ such that $\boldsymbol{\theta}_{T_0}>\delta_1$ on ${\rm Reg}_i T_1$, and a reflection argument and \eqref{eq:hopf for u}-\eqref{eq:doesnt hit subgraph} yield the corresponding upper bound $\boldsymbol{\theta}_{T_0} < 2\pi - \delta_2$ on ${\rm Reg}_i T_2$ for some $\delta_2>0$. Letting $\delta=\min\{\delta_1,\delta_2\}$ yields \eqref{eq:restricted theta range}.

Next, we perform a similar blowup procedure as in the previous step. 
For $r<1$, we define
\begin{align}\notag
    \boldsymbol{\theta}_r^{\rm inf} &= \inf \{\boldsymbol{\theta}_{T_0}(y) : y\in {\rm Reg}_i (T_0- z\a{M_1} ) \cap \bB_r\}\geq \delta \qquad \mbox{ and}\\ \notag
    \boldsymbol{\theta}_r^{\rm sup} &= \sup \{ \boldsymbol{\theta}_{T_0}(y) : y\in {\rm Reg}_i (T_0- z\a{M_1} \cap \bB_r\} \leq 2\pi - \delta\,.
\end{align}
Note that $\boldsymbol{\theta}_r^{\rm inf}$ is a decreasing function of $r$ and $\boldsymbol{\theta}_r^{\rm sup}$ is increasing, and so there are limits
\begin{equation*}
 \delta \leq \lim_{r\to 0}\boldsymbol{\theta}_r^{\rm inf}=:\boldsymbol{\theta}_{T_0}^{\rm min} \leq \boldsymbol{\theta}_{T_0}^{\rm max}:=\lim_{r\to 0}\boldsymbol{\theta}_r^{\rm sup}\leq 2\pi - \delta\,.
\end{equation*}
Arguing as in the previous step using Proposition \ref{lemma:hopf currents}, for any blowup $T_{00}$ of $T_0$ at $0$, it holds
\begin{equation}\notag
   M_1 \cup M_{e^{i\boldsymbol{\theta}_{T_0}^{\rm min}}} \cup M_{e^{i\boldsymbol{\theta}_{T_0}^{\rm max}}} \subset \spt T_{00}\,.
\end{equation}

We continue on iteratively, blowing up and identifying more and more halfplanes contained in further blowups. Since $T$, $T_0$, $T_{00}$, and so on all satisfy uniform upper mass bounds (Lemma \ref{lemma:local upper bf bounds}), this procedure must terminate after a finite number of blowups. The end result is a current which is supported on a finite number of oriented half hyperplanes and has constant multiplicity on each. By Lemma \ref{lemma: resolution in 2d}, this last blowup must be supported on the plane containing $M_1$ and this finishes the proof. \end{proof}

\begin{remark}
We notice that, if $0$ is a one-sided point for $T$, i.e., the density of $T$ at $0$ is close to $1/2$, we could have stopped the iteration in the very first step.
\end{remark}

\bibliographystyle{abbrv}
\bibliography{biblio}

\end{document}